\documentclass[12pt,reqno]{amsart}
\usepackage{amssymb}
\usepackage{amsmath, amssymb, amsthm}
\usepackage{extarrows}
\usepackage{paralist}
\usepackage{graphics}
\usepackage{epsfig}
\usepackage{psfrag}
\usepackage{subeqnarray}
\usepackage{cases}
\usepackage{cite}
\usepackage{bm}
\usepackage{cleveref}
\usepackage{bibentry}
\usepackage{appendix}
\usepackage{xcolor}
\usepackage{enumerate}
\usepackage{color}

\crefformat{section}{\S#2#1#3} 
\crefformat{subsection}{\S#2#1#3}
\crefformat{subsubsection}{\S#2#1#3}
\newtheorem{theorem}{Theorem}[section]
\newtheorem{definition}{Definition}[section]
\newtheorem{lemma}{Lemma}[section]

\newtheorem{problem}{Problem}[section]
\newtheorem{proposition}{Proposition}[section]
\numberwithin{figure}{section}
\numberwithin{equation}{section}

\newcommand{\dd}{{\rm d}}

\begin{document}
\title[Global Solutions of a $2$-D Riemann Problem]
{Global Solutions of \\ a Two-Dimensional Riemann Problem \\ for the Pressure Gradient System}

\author[G.-Q. Chen]{Gui-Qiang G. Chen}
\address
{Gui-Qiang G. Chen, Mathematical Institute, University of Oxford,
Oxford OX2 6GG, UK;
School of Mathematical Sciences, Fudan University, Shanghai200433, China;
AMSS,
Chinese Academy of Sciences, Beijing 100190, China}
\email{chengq@maths.ox.ac.uk}

\author[Q. Wang]{Qin Wang}
\address
{Qin Wang, Department of Mathematics, Yunnan University, Kunming, 650091, China}
\email{mathwq@ynu.edu.cn}

\author[S. Zhu]{Shengguo Zhu}
\address
{Shengguo Zhu, School of Mathematical Sciences,  Shanghai Jiao Tong University, Shanghai 200240, China}
\email{zhushengguo@sjtu.edu.cn}

\dedicatory{{\it Dedicated to Professor Shuxing Chen on the occasion of his 80th birthday}}

\date{}

\begin{abstract}
We are concerned with a two-dimensional ($2$-D) Riemann problem for compressible flows
modeled by the pressure gradient system that is a $2$-D hyperbolic system of conservation laws.
The Riemann initial data consist of four constant states in four sectorial regions such that
two shock waves and two vortex sheets are generated between the adjacent states.
This Riemann problem can be reduced to a boundary value problem in the self-similar coordinates with
asymptotic boundary data consisting of the two shocks, the two vortex sheets, and the four constant states,
along with two sonic circles determined by
the Riemann initial data, for a nonlinear system of mixed-composite type.
The solutions keep the four constant states and four planar waves outside the outer sonic circle in the self-similar
coordinates.
In particular, the two shocks keep planar until they meet the outer sonic circle at two different points
and then generate a diffracted shock to be expected to connect these two points,
whose exact location is {\it apriori} unknown which is regarded as a free boundary.
Then the $2$-D Riemann problem can be reformulated as a free boundary problem,
in which the diffracted transonic shock is the one-phase free boundary
to connect the two points, while the other part of the outer sonic circle forms
the part of the fixed boundary of the problem.
We establish the global existence of a solution of the free boundary problem,
as well as the $C^{0,1}$--regularity of both the diffracted shock across the two points and
the solution across the outer sonic boundary which is optimal.
One of the key observations here is that the diffracted transonic shock can not intersect
with the inner sonic circle in the self-similar coordinates.
As a result, this $2$-D Riemann problem is solved globally, whose solution contains two vortex sheets and
one global $2$-D shock connecting the two original shocks generated by the Riemann data.
\end{abstract}

\maketitle
\section{Introduction}
The two-dimensional (2-D) full Euler equations are of the conservation form:
\begin{equation}
\label{euler1}
\bm{U}_t+\text{div}_{\rm{x}}\bm{F}=0 \qquad\,\,\mbox{for $t\geq 0\,$ and $\,{\rm{x}}=(x_1,x_2)\in \mathbb{R}^2$},
\end{equation}
with
$$
\bm{U}:=(\rho,\rho \bm{u},\rho E),\qquad \bm{F}:=(\rho \bm{u},\rho\bm{u}\otimes \bm{u}+pI, (\rho E+p)\bm{u}),
$$
where  $\rho>0$ is the density, $\bm{u}=(u,v)$ the velocity, $p$ the pressure,
and
$$
E=\frac{|\bm{u}|^2}{2}+e
$$
represents the total energy per unit mass with the internal energy $e$ given by
$e=\frac{p}{(\gamma-1)\rho}$ for the adiabatic constant $\gamma>1$ for polytropic gases.

There are two mechanisms in the fluid motion: inertia and pressure difference.
Corresponding to a separation of these two mechanisms,
a natural flux-splitting of $\bm{F}$ is to divide it into
two parts:
$\bm{F}=\bm{F}_1+\bm{F}_2$ with
$$
\bm{F}_1:=(\rho \bm{u},\rho\bm{u}\otimes \bm{u}, \rho E \bm{u} ),\qquad \bm{F}_2:=(0,\,pI_{2\times 2},\, p\bm{u}),
$$
where $I_{2\times 2}$ is the diagonal identity matrix.
Correspondingly,  the Euler equations \eqref{euler1} can be split into two subsystems of conservation laws:
$$
\bm{U}_t+\text{div}\bm{F}_1=0, \qquad\,\,\,\, \bm{U}_t+\text{div}\bm{F}_2=0,
$$
which are called  the pressureless Euler system and the pressure gradient system, respectively;
also see \cite{li1985second,agarwal1994modified}.
Similar flux-splitting ideas have been widely used in order to design the so-called flux-splitting
schemes and their high-order accurate extensions. Many flux-splittings
have been derived in the literature for the compressible Euler equations of gas
dynamics and are currently used in fluid dynamics codes.
See \cite{ChenLeFloch,li1985second,agarwal1994modified} and the references cited therein.

In this paper, we focus on the pressure gradient system that is corresponding to flux $\bm{F}_2$.
The explicit form for the pressure gradient system is
\begin{equation}\label{euler2}
\begin{cases}
\rho_t=0,\\[2pt]
(\rho \bm{u})_t+\nabla_{\rm{x}} p=0,\\[2pt]
(\rho E)_t+\text{div}_{\rm{x}}(p \bm{u})=0.
\end{cases}
\end{equation}

By a suitable scaling in \eqref{euler2} and taking $\rho\equiv 1$,
the pressure gradient system is of the following form:
\begin{equation}\label{pgs}
\begin{cases}
u_t+p_{x_1}=0,\\[2pt]
v_t+p_{x_2}=0,\\[2pt]
E_t+(pu)_{x_1}+(pv)_{x_2}=0,
\end{cases}
\end{equation}
where  $E=\frac{|\bm{u}|^2}{2}+p$.
Furthermore, system \eqref{pgs} can be also deduced from the physical validity when the velocity
is small and the adiabatic gas constant $\gamma$ is large; see Zheng \cite{yuxi1997existence}.
An asymptotic derivation of system \eqref{pgs} has been also presented by Hunter as described in \cite{zheng2006two}.
We refer the reader to \cite{li1998two,zheng2012systems} for further background on system \eqref{pgs}.
Besides the pressure gradient system,
there are also several other important nonlinear partial differential equations (PDEs)
derived from the full Euler equations,
such as the potential flow equation that has been widely used in aerodynamics,
as well as the nonlinear wave system,
the unsteady transonic small disturbance equations,
and the pressureless Euler system as mentioned above;
see
\cite{chen2010global,chen2014shock,chen2018mathematics,kim2010global,vcanic2000free,keyfitz1998riemann}
and the references cited therein.
The analysis of these nonlinear PDEs has motivated and inspired the developments of
new techniques and ideas to deal with the corresponding problems
for the Euler equations.

The Riemann problem was first introduced by B. Riemann in 1860 in his pioneering work \cite{Riemann}
to analyze discontinuous solutions of the 1-D Euler equations for gas dynamics.
It is an initial value problem with the simplest discontinuous initial data,
which are scaling invariant and piecewise constant.
The Riemann solutions have played a fundamental role
in the mathematical theory of 1-D hyperbolic systems of conservation laws; see \cite{chang1989riemann,chen2018mathematics,dafermos2000,lax1957hyperbolic,glimm1965,smoller1994}
and the references cited therein.
The 2-D Riemann problem is substantially different and much more complicated than the 1-D case;
see \cite{ChangChenYang,chang1989riemann,SXChen1992,SXChen1997,Chen-Qu2012,li1998two,zheng2012systems}
and the references cited therein.

One of the prototypical 2-D Riemann problems is that the Riemann initial data consist of four different
constant states in the four quadrants so that there is only one wave that is generated between two adjacent states.
Each wave between any two adjacent states is of one of at least three types
of planar waves: shock wave, rarefaction wave, and vortex sheet.
Then the 2-D Riemann problem is to analyze the different combinations/interactions of these four waves
in a domain containing the origin.
The solutions of such a $2$-D Riemann problem for system \eqref{pgs} were analyzed
via the characteristic method and the corresponding numerical simulations were presented
in Zhang-Li-Zhang \cite{zhang1998two}.
It has been observed that the mathematical structure of the pressure gradient system is
strikingly in agreement to that of the Euler equations.
In \cite{li1998two,zhang1998two}, it was shown that there are twelve genuinely different
cases, besides three trivial cases, for the solutions of the $2$-D Riemann problem for the pressure gradient system.
To our knowledge, there have been few rigorous mathematical results on the global existence for
the non-trivial cases,
owing to lack of effective techniques for handling several main difficulties in
the analysis of nonlinear PDEs
such as equations of mixed elliptic-hyperbolic type, free boundary problems,
and corner singularity.

\begin{figure}[htbp]
\vspace{-0.3cm}
\includegraphics[scale=1.25]{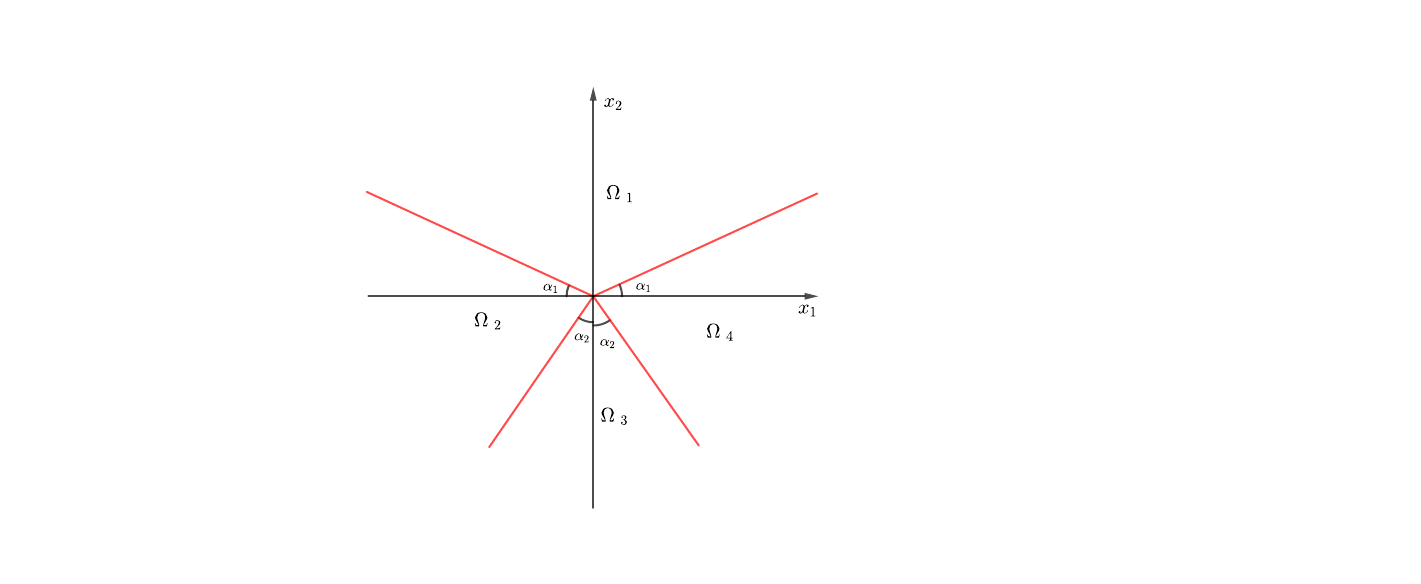}
\setlength{\abovecaptionskip}{-1.0cm}
\caption{The general Riemann initial data}\label{figure1}
\end{figure}

In this paper, we consider a more general Riemann problem
for system \eqref{pgs}, whose initial data consists of four constant states in four sectorial regions
$\Omega_i$  with symmetric sectorial angles (see Fig. \ref{figure1}):
\begin{equation}\label{initialdata}
(p,u,v)(0,x_1,x_2)=(p_i,u_i,v_i) \qquad\,\, \mbox{for $(x_1,x_2)\in \Omega_i, \, i=1,2,3,4$}.
\end{equation}
One of our motivations is to understand
the intersections of two shock waves and two vortex sheets.
For this purpose, the four initial constant states are required to satisfy the following conditions:
\begin{equation}\label{initialdate1}
\begin{cases}
\text{A forward shock $S_{41}^{+}$ is formed between states $(1)$ and $(4)$},\\[2pt]
\text{A backward shock $S_{12}^{-}$ is formed between states $(1)$ and $(2)$},\\[2pt]
\text{A vortex sheet $J_{23}^{+}$ is formed between states $(2)$ and $(3)$},\\[2pt]
\text{A vortex sheet $J_{34}^{-}$ is formed between states $(3)$ and $(4)$}.
\end{cases}
\end{equation}
These four waves can be obtained by solving four 1-D Riemann problems in the self-similar
coordinates $(\xi,\eta)=(\frac{x_1}{t},\frac{x_2}{t})$, which form the configuration
as shown in Fig. \ref{figure2}.

\vspace{-15pt}
\begin{figure}[htbp]
\vspace{-0.5cm}
\includegraphics[scale=1]{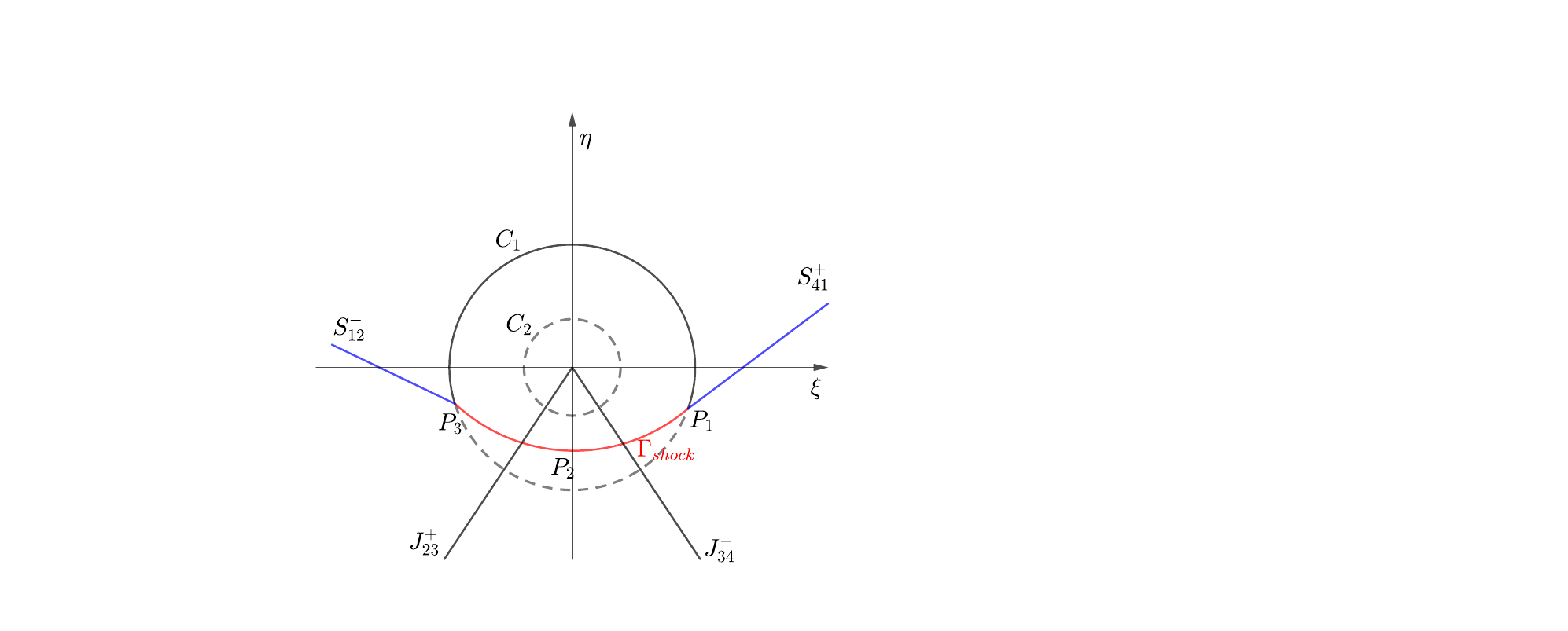}
\setlength{\abovecaptionskip}{-1.0cm}
\caption{The configuration of the four initial waves}\label{figure2}
\end{figure}

When the two shock waves $S^{-}_{12}$ and $S^{+}_{41}$ meet the outer sonic circle $C_1$ of state $(1)$,
shock diffraction occurs, and then $S^{-}_{12}$ and $S^{+}_{41}$ are expected to bend and form a diffracted shock,
denoted by $\Gamma_{\rm shock}$.
One of the main difficulties is that the location of $\Gamma_{\rm shock}$ is {\it apriori} unknown,
so that it is {\it apriori} unclear whether $\Gamma_{\rm shock}$ could intersect with the inner sonic circle $C_2$ of state $(2)$.
Zheng \cite{zheng2003global} studied this Riemann problem initially
with the assumption that angle $\alpha_1=\alpha_2$  is close to zero.
This assumption ensures that the two shocks bend slightly and the diffracted shock $\Gamma_{\rm shock}$ does
not meet the inner sonic circle $C_2$.
However, it has been an open problem when the angle between the two shocks is not close to $\pi$,
since the work of Zheng \cite{zheng2003global}.

The purpose of this paper is to
solve the Riemann problem globally for the general case
so that an affirmative answer to this open problem is provided.
In particular, we establish the global existence of entropy solutions for this Riemann problem allowing all
angles $\alpha_i\in (0,\frac{\pi}{2}), i=1,2$.
To solve this problem,
we first reformulate the problem as a free boundary problem
involving transonic shocks. Then we carefully establish the required appropriate
properties and uniform estimates of approximate and exact solutions so that the techniques developed
in Chen-Feldman \cite{chen2010global, chen2018mathematics} can be employed; also see
\cite{bae2009regularity,bae2019prandtl,chen2007stability,yuan2006transonic} and
the references cited therein.
This involves several core difficulties in the theory of the
underlying nonlinear PDEs: optimal estimates of solutions of nonlinear degenerate PDEs across the other sonic circle $C_1$
and corner
singularities (at corners $P_1$ and $P_3$ formed between the transonic shock as a free boundary and the sonic circle $C_1$),
in addition to the involved nonlinear PDE of mixed elliptic-hyperbolic type and free boundary
problem.

The organization of this paper is as follows: In \S \ref{section2},
we reformulate the Riemann problem into the free boundary problem and present our main results and strategies.
In \S \ref{sec3}, we give a complete proof of the global existence of solutions of the free boundary problem.
In \S \ref{sec4}, we obtain the optimal $C^{0,1}$--regularity of solutions near the degenerate sonic boundary $C_1$
and at corners $P_1$ and $P_3$.
Finally, in \S 5, we obtain the existence and regularity of global solutions of the 2-D Riemann problem of
the pressure gradient system \eqref{pgs}.

\section{Reformulation of the Riemann Problem and Main Theorem}\label{section2}
Based on the invariance of both the system and the Riemann initial data under the self-similar scaling,
we seek self-similar solutions in the self-similar coordinates.
For this purpose, in this section, we first reformulate the Riemann problem into a free boundary problem, present
the main results in the main theorem, Theorem 2.1, and then describe the strategies
to achieve them in \cref{sec2.3}--\cref{sec2.4}.

More precisely, we seek self-similar solutions with the form:
\begin{equation*}
(p,u,v)(t,x_1,x_2)=(p,u,v)(\xi,\eta) \qquad \text{with $(\xi,\eta)=(\frac{x_1}{t},\frac{x_2}{t}),\,t>0$}.
\end{equation*}
In the $(\xi,\eta)$--coordinates, system \eqref{pgs} can be rewritten as
\begin{equation}
\label{pgs2}
\begin{cases}
(\xi u)_\xi+ (\eta u)_\eta - p_\xi -2u=0, \\[1mm]
(\xi v)_\xi + (\eta v)_\eta -p_\eta -2v=0, \\[1mm]
(\xi E)_\xi +(\eta E)_\eta - (pu)_\xi- (pv)_\eta -2E=0.
\end{cases}
\end{equation}

\subsection{Shock waves and vortex sheets in the self-similar coordinates} \label{sec2.1}

Let $\eta=\eta(\xi)$ be a $C^1$--discontinuity curve of a bounded discontinuous solution of
system \eqref{pgs2}.
From the Rankine-Hugoniot relation on $\eta(\xi)$:
\begin{equation*}
\begin{cases}
(\xi\sigma-\eta)[u]-\sigma[p]=0,\\[1mm]
(\xi\sigma-\eta)[v]+[p]=0,\\[1mm]
(\xi\sigma-\eta)[E]-\sigma[pu]+[pv]=0,
\end{cases}
\end{equation*}
we find either the nonlinear discontinuities:
\begin{equation}\label{shock11}
\begin{cases}
\frac{\dd\eta}{\dd\xi}=\sigma_{\pm}=-\frac{[u]}{[v]}
  =\frac{\xi\eta\pm\sqrt{\overline{p}(\xi^2+\eta^2-\overline{p})}}{\xi^2-\overline{p}},\\[1.5mm]
[p]^2=\overline{p}([u]^2+[v]^2),
\end{cases}
\end{equation}
or linear discontinuity:
\begin{equation}\label{shock12}
\begin{cases}
\sigma_{0}=\frac{\eta}{\xi}=\frac{[v]}{[u]},\\[5pt]
[p]=0,
\end{cases}
\end{equation}
where $\overline{p}$
is the average of the pressure
on the two sides of the discontinuity,
and $[w]$ denotes the jump of $w$ across the discontinuity.

A discontinuity is called a shock if it satisfies \eqref{shock11} and the entropy
condition --- the pressure $p$ increases across it in the flow direction; that is,
the pressure on the wave front is larger than that on the wave back.
The shock is of two types, $S^{\pm}$:
\begin{itemize}
\item$ S=S^{+}$ if $\nabla_{(\xi,\eta)}p$ and the flow direction form a right-hand system;

\smallskip
\item $S=S^{-}$ if $\nabla_{(\xi,\eta)}p$ and the flow direction form a left-hand system.
\end{itemize}
A discontinuity is called a vortex sheet if it satisfies \eqref{shock12}.
A vortex sheet is of two types according to the sign of the vorticity:
$$
J^{\pm}: \quad \text{curl}(u,v)=\pm\infty.
$$

\subsection{Reformulation of the Riemann problem into
            a free boundary problem} \label{sec2.2}

We first show the following lemma:

\begin{lemma}\label{le2.1}
For fixed $(p_1,u_1,v_1)$ and $p_2=p_3=p_4$ satisfying $p_1>p_2$,
there exist states $(u_i,v_i), i=2,3,4$, such that the conditions
in \eqref{initialdate1} for the Riemann initial data hold,
and $(u_i,v_i), i=2,3,4$, depend on
angles $(\alpha_1,\alpha_2)$ continuously.
\end{lemma}

\begin{proof}
For given $(p_1,u_1,v_1)$ and $p_1>p_2$, we first consider $(u_2,v_2)$.
Since the Rankine-Hugoniot conditions on  $S_{12}^{-}$ hold:
\begin{equation*}\label{state2}
\begin{cases}
[u\sin\alpha_1+v\cos\alpha_1]=-\frac{[p]}{\sqrt{\overline{p}}},\\[1mm]
[-u\cos\alpha_1+v\sin\alpha_1]=0,
\end{cases}
\end{equation*}
a direct computation shows that
\begin{equation*}
S_{12}^{-}:=\big\{(\xi,\eta)\,:\, \xi\sin\alpha_1+\eta\cos\alpha_1=-\sqrt{\overline{p}}\big\}
\end{equation*}
with
\begin{equation*}\label{state21}
(u_2, v_2)=(u_1, v_1)-\frac{[p]}{\sqrt{\overline{p}}}(\sin\alpha_1,\cos\alpha_1).
\end{equation*}

Next, we turn to $(u_4,v_4)$. The Rankine-Hugoniot conditions on $S_{41}^{+}$ are:
\begin{equation*}\label{state4}
\begin{cases}
[u\sin\alpha_1-v\cos\alpha_1]=\frac{[p]}{\sqrt{\overline{p}}},\\[1mm]
[u\cos\alpha_1+v\sin\alpha_1]=0,
\end{cases}
\end{equation*}
which imply
\begin{equation*}
S_{41}^{+}:=\big\{(\xi,\eta)\,:\,  \xi\sin\alpha_1-\eta\cos\alpha_1=\sqrt{\overline{p}}\big\}
\end{equation*}
with
\begin{equation*}
\label{state41}
(u_4, v_4)=(u_1, v_1)+\frac{[p]}{\sqrt{\overline{p}}}(\sin\alpha_1, -\cos\alpha_1).
\end{equation*}

Finally, we consider $(u_3,v_3)$. To guarantee the existence of two vortex sheets
$J_{23}^{+}$ and $J_{34}^{-}$, we have
\begin{equation*}\label{state3}
\begin{cases}
u_2\cos\alpha_2-v_2\sin\alpha_2=u_3\cos\alpha_2-v_3\sin\alpha_2,\\[1mm]
u_4\cos\alpha_2+v_4\sin\alpha_2=u_3\cos\alpha_2+v_3\sin\alpha_2.
\end{cases}
\end{equation*}
Solving $(u_3,v_3)$ from the above two equations, we obtain
\begin{equation*}
\label{state31}
(u_3, v_3)=(u_1, v_1)=\frac{[p]}{\sqrt{\overline{p}}}(0, \frac{\sin\alpha_1}{\sin\alpha_2}-\cos\alpha_1).
\end{equation*}
It is clear that $(u_i,v_i), i=2,3,4$, depend on angles $(\alpha_1, \alpha_2)$ continuously.
\end{proof}

\begin{figure}[htbp]
\vspace{-1cm}
\centering{\includegraphics[scale=0.8]{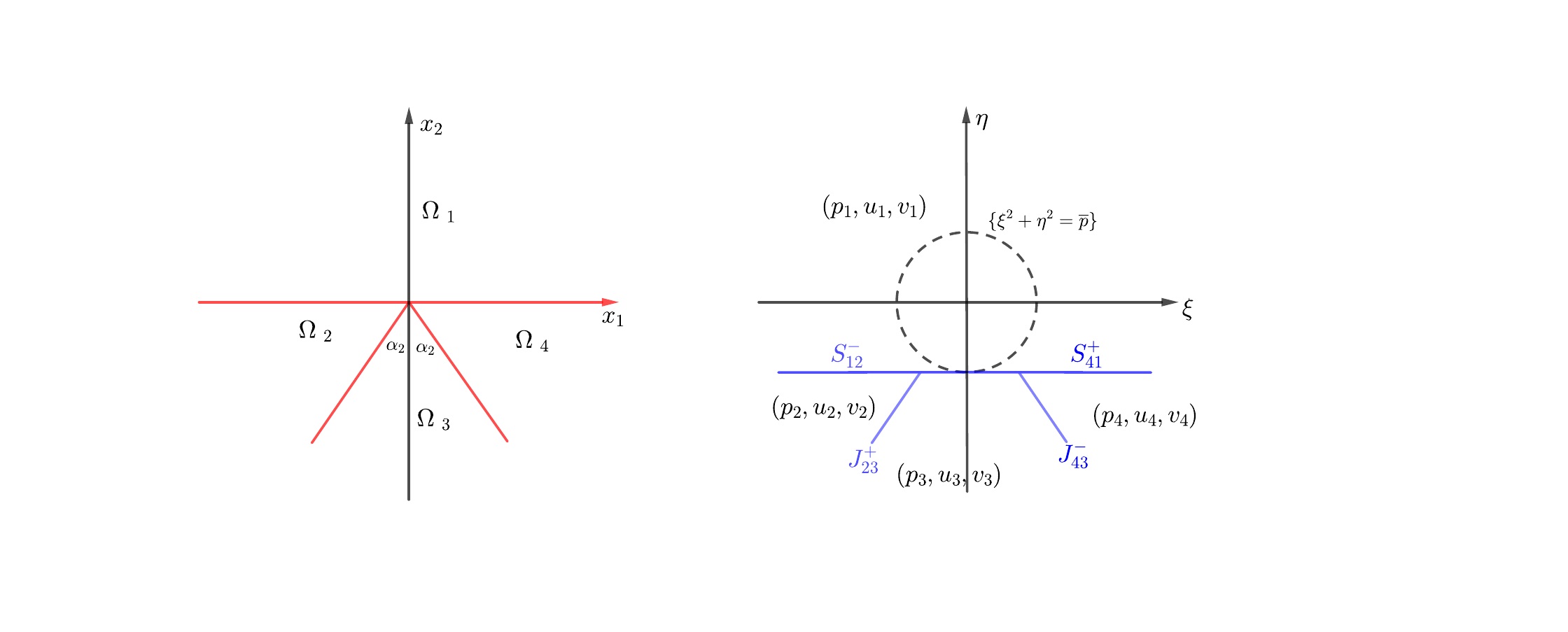}  }
\parbox{16cm}{\setlength{\abovecaptionskip}{-1.5cm}\caption{\label{figure3}
{The Riemann data and the global solution when $\alpha_1=0$}}}
\end{figure}

There is a critical case when $\alpha_1=0$.
Then the Riemann initial data satisfy
$$
p_1>p_2=p_3=p_4, \quad u_1=u_2=u_3=u_4,\quad v_1>v_2=v_3=v_4.
$$
The global Riemann solution is piecewise constant with two planar shocks:
\begin{equation*}
S_{12}^{-}/S_{41}^{+}: \begin{cases}
\,\eta=-\sqrt{\overline{p}}
\qquad \mbox{with $\overline{p}=\frac{p_1+p_2}{2}$},
\\[1mm]
\,[v]=-\frac{[p]}{\sqrt{\overline{p}}}, \qquad\, [u]=0,\\[1mm]
\xi<0 \,\,\text{for}\,\, S_{12}^{-},\qquad \,\,  \xi>0 \,\,\text{for}\,\, S_{41}^{+},
\end{cases}
\end{equation*}
and two characteristic lines $J_{23}^{+}$ and $J_{34}^{-}$ (reduced from the two vortex sheets),
as shown in Fig. \ref{figure3}.
The two planar shocks $S_{12}^{-}$ and $S_{41}^{+}$ are both tangential to the circle, $\{\xi^2+\eta^2=\overline{p}\}$,
with the tangent point on the circle as the end-point.
From the expression of $J_{23}^{+}$ given in \eqref{shock12},
we know that $p_2=p_3$ on both sides of $J_{23}^{+}$.
At the point where $J_{23}^{+}$ intersects with $S_{12}^{-}$,
we deduce that $J_{23}^{+}$ does not effect the shock owing to $p_2=p_3$.
The intersection between $J_{34}^{-}$ and $S_{41}^{+}$ can be handled
in the same way.

\medskip
In the following, we focus on the case that $\alpha_1\in (0,\frac{\pi}{2})$, for which
we want to solve.

\smallskip
From \eqref{pgs2}, we can derive a second-order nonlinear equation for $p$:
\begin{equation}\label{eqforp}
(p-\xi^{2})p_{\xi\xi}-2\xi\eta p_{\xi\eta}+(p-\eta^{2})p_{\eta\eta}+\frac{(\xi p_{\xi}+\eta p_{\eta})^{2}}{p}
-2(\xi p_{\xi}+\eta p_{\eta})=0.
\end{equation}
It is easy to verify that equation \eqref{eqforp} is of mixed hyperbolic-elliptic
type, which is hyperbolic when $\xi^2+\eta^2>p$ and elliptic when $\xi^2+\eta^2<p$.
The sonic circle is: $\xi^2+\eta^2=p$.

Furthermore, in the polar coordinates:
$$
(r,\theta)=(\sqrt{\xi^2+\eta^2}, \arctan(\frac{\eta}{\xi})),
$$
equation \eqref{eqforp} becomes
\begin{equation}
\label{eqforp2}
Qp:=(p-r^{2})p_{rr}+\frac{p}{r^{2}}p_{\theta\theta}+\frac{p}{r}p_{r}+\frac{1}{p}(r p_{r})^{2}-2r p_{r}=0,
\end{equation}
which is  hyperbolic when $p<r^2$ and elliptic when $p>r^2$.
The sonic circle is given by $r=r(\theta)$ satisfying that
$r^2(\theta)=p(r(\theta),\theta)$.

In the $(\xi,\eta)$--coordinates,
the four elementary waves come from the far-field (at infinity corresponding to $t=0$)
and keep planar waves before the two shocks meet the outer sonic circle $C_1$
of state $(1)$:
$$
C_1:=\{(\xi,\eta)\,:\,\xi^2+\eta^2=p_1\}.
$$
When the two shocks $S_{12}^{-}$ and $S_{41}^{+}$ meet the sonic circle $C_1$
at points $P_3$ and $P_1$, respectively,
our main concern is whether they bend and meet to form a diffracted shock, denoted by $\Gamma_{\rm shock}$;
see Fig. \ref{figure2}.
Since the whole configuration is symmetric with respect to the
$\eta$--axis,
we infer that $\Gamma_{\rm shock}$ must be vertical to $\xi=0$ at point $P_2$,
where the two diffracted shocks meet.
We should point out here particularly that the two vortex sheets
$J_{23}^{+}$ and $J_{34}^{-}$ and the diffracted shock $\Gamma_{\rm shock}$
have no influence each other during the intersection.
Therefore, from now on, we {\it ignore} the two vortex sheets first and focus only on the diffracted shock.

\smallskip
Moreover, we remark that, at this point,  we have not excluded the case that the diffracted shock may
degenerate partially into a portion of the inner sonic circle $C_2$ of state $(2)$.
Once this case occurs, $p=p_2$ on the sonic circle.
It will be seen that $p=p_2$  satisfies the oblique derivative conditions on
the diffracted shock automatically.

\smallskip
On $\Gamma_{\rm shock}$, the Rankine-Hugoniot conditions in the polar coordinates
must be satisfied:
\begin{equation}\label{eq:2.1}
\begin{cases}
r[u]-[p]\cos\theta=\frac{1}{r}\frac{\dd r}{\dd\theta}\sin\theta [p],\\[1mm]
r[v]-[p]\sin\theta=-\frac{1}{r}\frac{\dd r}{\dd\theta}\cos\theta[p],\\[1mm]
r[E]-[pu]\cos\theta-[pv]\sin\theta
=\frac{1}{r}\frac{\dd r}{\dd\theta}\big(\sin\theta [pu]-\cos\theta [pv]\big),
\end{cases}
\end{equation}
where $[w]$ denotes the jump of $w$ across $\Gamma_{\rm shock}$.
Owing to
$$
[pu]=\overline{p}\,[u]+\overline{u}\,[p],
$$
with $\overline{p}$ as the average of the two neighboring
states of $p$,
we eliminate $[u]$ and $[v]$ in the third equation in \eqref{eq:2.1} to obtain
\begin{equation*}
\Big(\frac{\dd r}{\dd\theta}\Big)^2=\frac{r^2(r^2-\overline{p})}{\overline{p}}.
\end{equation*}

The shock diffraction can be also considered to be created from point $P_2$ in two directions,
which implies that $r^{\prime}(\theta)>0$ for $\theta\in [\frac{3\pi}{2},\theta_{1}]$,
and $r^{\prime}(\theta)<0$ for $\theta\in [\theta_{3},\frac{3\pi}{2}]$,
where $\theta_{i}$ are denoted as the $\theta$--coordinates of points $P_i$, $i=1,3$,
respectively.
Thus, we choose
\begin{equation}\label{shockeq}
\frac{\dd r}{\dd\theta}=g(p(r(\theta),\theta),r(\theta)):=
\begin{cases}
r\sqrt{\frac{r^2-\overline{p}}{\overline{p}}}\qquad &\text{for}\,\, \theta\in [\frac{3\pi}{2},\theta_{1}],\\[5pt]
-r\sqrt{\frac{r^2-\overline{p}}{\overline{p}}}\qquad  &\text{for}\,\, \theta\in [\theta_{3},\frac{3\pi}{2}].
\end{cases}
\end{equation}
Moreover, it follows from \eqref{eq:2.1} that
\begin{equation}\label{eq:2.6}
[p]^2=\overline{p}\,\big([u]^2+[v]^2\big).
\end{equation}
From \eqref{eq:2.6},  taking the derivative $r^\prime\partial_r+\partial_\theta$
along the shock yields the derivative boundary condition on $\Gamma_{\rm shock}$:
\begin{equation}\label{eqrh}
\mathcal{B}p=\sum\limits_{i=1}^{2}\beta_i D_ip:=\beta_1 p_r+\beta_2 p_\theta=0,
\end{equation}
where $\mathbf{\beta}=(\beta_1,\beta_2)$ is a function of $(p,p_2,r(\theta),r^\prime(\theta))$ with
\begin{equation}
\begin{split}\label{eqrh2}
\beta_1&=2r^\prime\Big(\frac{r^2-\overline{p}}{r^2}-\frac{[p]}{4\overline{p}}
+\frac{\overline{p}(r^2-p)}{r^2p}\Big), \\
\beta_2&=\frac{4(r^2-\overline{p})}{r^2}-\frac{[p]}{2\overline{p}}.
\end{split}
\end{equation}
The obliqueness becomes
$$
(\beta_1,\beta_2)\cdot (1,-r^\prime(\theta))=-2r^\prime(\theta)\big(1-\frac{\overline{p}}{p}\big)
\equiv: \mu.
$$
Note that $\mu$ vanishes at point $P_2$  where $r^\prime(\frac{3\pi}{2})=0$.
When the obliqueness fails, we have
$$
\beta_1=0,\qquad \beta_2=-\frac{[p]}{2\overline{p}}<0,
$$
owing to  $p>p_2$.

Let $\Gamma_{\rm sonic}$ be the larger portion $\widehat{P_1P_3}$ of the sonic circle $C_1$
of state $(1)$. On $\Gamma_{\rm sonic}$, $p$ satisfies the Dirichlet boundary condition:
\begin{equation}\label{dirichlet}
p=p_1.
\end{equation}
Let $\Omega$ be the bounded domain enclosed by $\Gamma_{\rm sonic}$
and $\Gamma_{\rm shock}$.

\begin{problem}[Free boundary value problem]\label{problem}
\emph{The Riemann problem for the pressure gradient system \eqref{pgs} with the Riemann initial data
satisfying \eqref{initialdate1}  can be reformulated into
the following free boundary value problem}:
\begin{equation*}
\begin{cases}
\text{\emph{Equation} \eqref{eqforp2}}\, &\qquad \text{\emph{in}}\,\,\Omega,\\[3pt]
\text{\emph{The derivative boundary conditions} \eqref{eqrh}--\eqref{eqrh2}}\, &\qquad\text{\emph{on}}\,\, \Gamma_{\rm shock},\\[3pt]
\text{\emph{The Dirichlet boundary condition}}\, \eqref{dirichlet}
&\qquad\text{\emph{on}}\,\,  \Gamma_{\rm sonic},
\end{cases}
\end{equation*}
\emph{where $\Gamma_{\rm shock}$ is a free boundary to be determined as given by \eqref{shockeq}}.
\end{problem}

\smallskip
\subsection{Main theorem}\label{sec2.3}
We now state our main theorem of this paper.

\begin{theorem}\label{thm1}
There exists a global solution $p(r,\theta)$ of Problem {\rm \ref{problem}}
in domain $\Omega$ with the free boundary
$r=r(\theta), \theta\in [\theta_{3},\theta_{1}]$, such that
\begin{equation*}
p\in C^{2,\alpha}(\Omega)\cap C^{\alpha}(\overline{\Omega}), \qquad
r\in C^{2,\alpha}((\theta_{3},\theta_{1}))\cap C^{1,1}([\theta_{3},\theta_{1}]),
\end{equation*}
where $\alpha\in (0,1)$ depends only on the Riemann initial data.
Moreover, the global solution $(p(r,\theta),r(\theta))$
satisfies the following properties{\rm :}
\begin{itemize}
\item[\emph{(i)}] $p>p_2$ on $\Gamma_{\rm shock}$, that is, the diffracted
shock $\Gamma_{\rm shock}$ does not meet the sonic circle $C_2$ of state $(2)${\rm ;}

\item[\emph{(ii)}] The shock curve $\Gamma_{\rm shock}$ is strictly convex in the self-similar coordinates{\rm ;}

\item[\emph{(iii)}] The global solution $p(r,\theta)$ is $C^{0,\alpha}$ up to the sonic boundary
$\Gamma_{\rm sonic}$ and Lipschitz continuous across $\Gamma_{\rm sonic}${\rm ;}

\item[\emph{(iv)}] The Lipschitz regularity of the solution across $\Gamma_{\rm sonic}$
from the inside of the subsonic region is optimal.
\end{itemize}
\end{theorem}

\subsection{Main strategies}\label{sec2.4}
\label{strategy}
There are three main difficulties in establishing the existence of
solutions of Problem \ref{problem}:
\begin{enumerate}[(i)]
\item  On the sonic boundary $\Gamma_{\rm sonic}$, owing to $p_1=r^2$, the ellipticity of equation \eqref{eqforp2} degenerates;
\item At point $P_2$ where the diffracted shock $\Gamma_{\rm shock}$ meets the $\eta$--axis $\xi=0$,
    the obliqueness of derivative boundary conditions fails, since
     $$
      (\beta_1,\beta_2)\cdot (1,-r^{\prime}(\theta))=0;
     $$
\item The diffracted shock $\Gamma_{\rm shock}$ is a free boundary, which may coincide with the sonic circle $C_2$ of state $(2)$.
\end{enumerate}

In the proof of the existence result, we first assume that $p\geq p_2+\delta$ holds on $\Gamma_{\rm shock}$ for some $\delta>0$,
which means that $\Gamma_{\rm shock}$ cannot coincide with the sonic circle $C_2$
of state $(2)$. This fact is eventually true and will be proved in \S \ref{sec4}.
For the second difficulty,  we may express this as a one-point Dirichlet condition $p(P_2)=\hat{p}$
by solving
$$
2r(\theta_2)=p(r(\theta_2),\theta_2)+p_2.
$$

We now illustrate a sketch of the proof for the existence of
solutions in the $(r,\theta)$--coordinates established in \S \ref{sec3}.
We divide the existence proof into four steps:

\begin{enumerate}[Step 1.]
\item Since equation \eqref{eqforp2} degenerates on the sonic boundary, we consider the regularized operator:
$$
Q^{\varepsilon}=Q+\varepsilon\Delta _{(\xi,\eta)}.
$$
We first fix a diffracted shock boundary $\Gamma_{\rm shock}$,
and then linearize the equation and the derivative boundary condition.
We employ the techniques developed in a series of works in \cite{lieberman1985perron,chen2014shock, kim2010global,zheng2003global}
to establish the existence result for the linear fixed mixed-type boundary problem for the regularized equation
in the polar coordinates.

\smallskip
\item Based on the estimates of solutions to the linear fixed boundary problem obtained in Step $1$,
 we prove the existence of a solution of the nonlinear fixed boundary problem via
 the Schauder fixed point theorem.

\smallskip
\item We apply the Schauder fixed point theorem again to obtain the existence of a solution of the free boundary problem
with the oblique derivative boundary condition for the regularized elliptic equation.
We conclude that the diffracted shock never meets the sonic circle $C_2$
of state $p_2$.

\smallskip
\item Finally, we study the limiting solution as the elliptic regularization parameter $\varepsilon$ tends to $0$
and complete the proof of the existence of solutions of Problem \ref{problem}.
\end{enumerate}

In \S \ref{sec4},  we introduce the
new coordinates $(x,y)=(r_1-r,\theta-\theta_1)$, which can flatten the sonic boundary.
It is shown that the optimal regularity of solutions across the sonic boundary is of $C^{0,1}$--regularity.
The most interesting point is the position of the diffracted shock, which is a free boundary.
This kind of free boundary problems occurs in many applications,
such as the shock reflection-diffraction
problem \cite{zheng2003global,vcanic2000free,canic2001free,chen2010global,chen2018mathematics,chen2014shock,kim2010global},
the Prandtl-Meyer shock configuration problem \cite{elling2008supersonic,bae2019prandtl}, among others.
In \S \ref{sec6}, we establish a corresponding theorem for the existence and regularity
of solutions of the $2$-D Riemann problem for the pressure gradient system \eqref{pgs}.

\section{Global Existence of Solutions of the Free Boundary Problem}\label{sec3}

In this section, we follow the strategies introduced in \S \ref{strategy}
to obtain the global existence of a solution of the free boundary problem, Problem \ref{problem}.
We first introduce the weighted norm used in this paper.

Let $\Omega^{\prime}:=\overline{\Omega}\backslash \overline{\Gamma_{\rm shock}}$.
For $P\in \{P_1,P_3\}$, we introduce the
corner regions:
$$
\Omega_{P}(\sigma):=\{x\in \Omega\,:\,\text{dist}(x,P)\leq \sigma\},\qquad
\Omega_V(\sigma):=\Omega_{p_1}(\sigma)\cup \Omega_{p_3}(\sigma).
$$
Define
\begin{align*}
&\Gamma^\prime(\sigma):=\{P\in \Gamma_{\rm shock}\,:\,\text{dist}(P,P_1)>\sigma,\, \text{dist}(P,P_3)>\sigma\},\\
&\Gamma(\sigma):= \Omega\cap\big(\cup_{P\in\Gamma^\prime(\sigma)}B_\sigma(P)\big),
\end{align*}
where $B_\sigma(P)$ is a ball of radius $\sigma$ centered at $P$.
Hence, $\Gamma(\sigma)$ is a region that
is close to $\Gamma_{\rm shock}$, but does not contain corners $P_1$ and $ P_3$.

We introduce the weighted norm
\begin{equation}
\label{funcspace}
\|u\|_a^b:=\sup\limits_{\sigma>0}
  \,\big\{\sigma^{a+b}\|u\|_{a,\overline{\Omega}\backslash(\Gamma(\sigma)\cup\Omega_{V}(\sigma))}\big\}
  \qquad\,\mbox{for any $a>0$ and $a+b\geq 0$}.
\end{equation}
The set of functions with finite norm $\|\cdot\|_{a}^{b}$  is denoted by $C_{b}^{a}(\Omega)$.

We now prove the existence of a solution of Problem 2.1 in the following four subsections.

\subsection{Regularized linear fixed boundary value problem}
For a fixed $\varepsilon\in(0,1)$, we consider the regularized
operator $Q^{\varepsilon}:=Q+\varepsilon\Delta_{(\xi,\eta)}$.
The equation for $p$ in the subsonic region  is
\begin{equation}
\label{regularizedforp}
Q^{\varepsilon}p:=(p-r^2+\varepsilon)p_{rr}+\frac{p+\varepsilon}{r^2}p_{\theta\theta}
+\frac{p+\varepsilon}{r}p_r+\frac{1}{p}(rp_r)^2-2rp_r=0.
\end{equation}
Since the position of the free boundary is not known {\it apriori},
we impose a cut-off function in \eqref{regularizedforp}.
Let $\zeta(\cdot)\in C^{\infty}(\mathbb{R})$ such that
\begin{equation}
\zeta(s)
=\begin{cases}
s\qquad &\text{if}\,\,s\geq 0,\\[1pt]
-\frac{\varepsilon}{2}\qquad &\text{if}\,\,s<-\varepsilon,
\end{cases}
\end{equation}
and $|\zeta^{\prime}(s)|\leq 1$.
We then consider the following modified governing equation:
\begin{equation}
\label{regularizedforp1}
Q^{\varepsilon,+}p:=\big(\zeta(p-r^2)+\varepsilon\big)p_{rr}+\frac{p+\varepsilon}{r^2}p_{\theta\theta}
+\frac{p+\varepsilon}{r}p_r+\frac{1}{p}(rp_r)^2-2rp_r=0.
\end{equation}

We define the iteration set
$\mathcal{R}^\varepsilon$ for shock $r(\theta)$,  which is a closed, convex subset of the H\"{o}lder
space $C^{1,\alpha}([\theta_3,\theta_1])$ as follows:

\begin{definition}
Let $r_i=\sqrt{p_i}$ be the radius of the sonic circle $C_i$ of state $p_i$, $i=1,2$.
The iteration set $\mathcal{R}^\varepsilon$ consists of elements $r(\theta)$ such that
\begin{enumerate}
\item[\emph{(R1)}]  $r(\theta_1)=r(\theta_3)=r_1;$

\smallskip
\item [\emph{(R2)}] $r_2 +\delta\leq r(\theta)\leq r_1$  for  all $\theta\in[\theta_3,\theta_1];$

\smallskip
\item[\emph{(R3)}]  $0< r^\prime(\theta)\leq r^\prime(\theta_1)$ for $\theta\in(\frac{3\pi}{2},\theta_1]$,
$r^\prime(\theta_3)\leq r^\prime(\theta)<0 $ for $\theta\in[\theta_3,\frac{3\pi}{2})$,
and $r^{\prime}(\frac{3\pi}{2})=0.$
\end{enumerate}
\end{definition}

In order to linearize the equation and the boundary conditions,  we define a function space $\mathcal{W}$.
\begin{definition}
The function space $\mathcal{W}\subset C^2_{(-\gamma_1)}$ consists of elements $\omega$ such that
\begin{enumerate}
\item[\emph{(W1)}] $p_2<\hat{p}^\varepsilon\leq \omega\leq p_1$, $\omega(P_2)=\hat{p}^\varepsilon$ and $\omega=p_1$ on $\Gamma_{\rm sonic};$

\smallskip
\item[\emph{(W2)}]  $\max\{\|\omega\|_{\alpha_0}, \|\omega\|_{1+\mu,\Gamma(d_0)}\}\leq K_0;$

\smallskip
\item[\emph{(W3)}]  $\|\omega\|_2^{(-\gamma_1)}\leq K_1$.
\end{enumerate}
The values of $\gamma_1,\alpha_0,\mu\in (0,1)$, and constants $d_0$, $K_0$, and $K_1$
will be specified later. The function set $\mathcal{W}$ is clearly closed, bounded, and convex.
\end{definition}

For a given $r(\theta)\in \mathcal{R}^\varepsilon$, let $\Gamma_{\rm shock}^{\varepsilon}$ be the fixed shock defined by
$$
\Gamma_{\rm shock}^{\varepsilon}:=\{(r(\theta),\theta)\,:\,\theta_3\leq\theta\leq\theta_1\}.
$$
The nonlinear equation \eqref{regularizedforp} and the boundary conditions \eqref{eqrh}--\eqref{eqrh2} are
now replaced by the linearized equation:
\begin{equation}\label{linearforp}
 L^{\varepsilon,+} p=\big(\zeta(\omega-r^2)+\varepsilon\big)p_{rr}+\frac{\omega+\varepsilon}{r^2}p_{\theta\theta}
+\frac{\omega+\varepsilon}{r}p_r+\frac{r^2\omega_r}{\omega}p_r-2rp_r=0,
\end{equation}
and the linearized oblique derivative boundary condition on $\Gamma^\varepsilon_{\rm shock}$:
\begin{equation}\label{eqrh3}
\mathcal{B}p:=\beta_1(\omega)p_r+\beta_2(\omega)p_\theta=0,
\end{equation}
with  $\omega\in \mathcal{W}$, and $\beta_i,\, i=1,2$, given in \eqref{eqrh2}.  Because of the bound
of $(W1)$,  equation \eqref{linearforp} is uniformly elliptic in $\Omega$ with ellipticity ratio depending on
the Riemann initial data and $\varepsilon$.

The other boundary condition is
\begin{equation}\label{eqcond}
p=p_1 \qquad \text{on $\Gamma_{\rm sonic}$}.
\end{equation}

Now we consider the following mixed-type boundary value problem for linear elliptic equation.

\begin{problem}[Linear fixed boundary value problem]\label{pr3-1}
Seek a solution $p$ of the linear elliptic equation \eqref{linearforp},
satisfying the derivative boundary condition \eqref{eqrh3} on given $\Gamma_{\rm shock}^{\varepsilon}$
and the Dirichlet boundary condition \eqref{eqcond} on $\Gamma_{\rm sonic}$.
\end{problem}

There have been several papers on the tangential oblique derivative problems for linear equations;
see \cite{egorov1969oblique,popivanov1997degenerate,hormander1966pseudo,winzell1981boundary} and the references cited therein.
However, we can not apply them directly  because the obliqueness of the derivative boundary condition fails at point $P_2$.
The main point is to find a way to remove the degeneracy.
We have the following result.

\begin{lemma}[Existence of solutions of Problem \ref{pr3-1}]\label{le2-1}
Assume that $\Gamma_{\rm shock}^{\varepsilon}$ is given by
$r(\theta)\in \mathcal{R}^{\varepsilon}$ for some $\alpha_1\in(0,1)$, and
$\omega\in\mathcal{W}$ for given $\alpha_0$, $\gamma_1$, $d_0$, $K_0$, and $K_1$.
Then there exist $\gamma_V,\alpha_\Omega\in(0,1)$  depending on $\varepsilon$, but independent of
$\alpha_1$ and $\gamma_1$, such that there is a solution
\begin{equation}
p^\varepsilon \in C^{1}(\overline{\Omega})\cap C^{2,\alpha}(\Omega^\prime)
\cap C^\gamma(\Omega_{V}(d))
\end{equation}
of Problem {\rm \ref{pr3-1}} for any $\alpha\leq \alpha_{\Omega}$,
$\gamma\leq \gamma_V$, and $d\leq d_0$.
Furthermore, solution $p^\varepsilon$ satisfies the following estimates{\rm :}
\begin{align*}
&\|p^\varepsilon\|_{\gamma,\Omega_{V}(d)}\leq M_1\|p^\varepsilon\|_0\qquad\,
\text{for any $\gamma\leq\gamma_V$},
\\
&\|p^\varepsilon\|_{2,\alpha;\Omega^\prime_{\rm loc}}\leq M_2\|p^\varepsilon\|_0
\qquad\, \text{for any $\alpha\leq\alpha_\Omega$},
\end{align*}
where $M_1$ is independent of $K_0$ and $K_1$,
and $M_2$ is independent of $K_1$ but depends on $K_0$.
\end{lemma}

\begin{proof}
It suffices to prove the local existence at point $P_2$, where the obliqueness of the derivative boundary condition fails.
Let $B$ be a sufficiently small neighborhood of $P_2$ with smooth boundary.
Let $L_{\sigma}$ be the line with $\sigma$--distance from point $P_2$ upward.
Let $\Omega_{\sigma}$ be the domain enclosed by $\partial B$, $\Gamma^{\varepsilon}_{\rm shock}$, and $L_{\sigma}$.
Now we consider the following boundary value problem:
\begin{equation}\label{eq3.9}
\begin{cases}
Q^{\varepsilon,+}p=0\qquad & \text{in $\Omega_{\sigma}$}, \\
\mathcal{B}p=0 \qquad &\text{on $\partial\Omega_{\sigma}\cap\Gamma^{\varepsilon}_{\rm shock}$},\\
p=h &\text{on $\partial B \cap \Omega$},\\
p=\hat{p}^{\varepsilon} &\text{on $L_{\sigma}$},
\end{cases}
\end{equation}
where $h$ is a smooth function satisfying that $\hat{p}^{\varepsilon}<h\leq p_1$.
Following \cite{lieberman1985perron},
there exists a solution
$$
p_{\sigma}\in C(\overline{\Omega}\cap \overline{\hat{B}})\cap C^{2,\alpha}(\Omega_{\sigma}\cap \hat{B})
$$
for a smaller neighborhood $\hat{B}$ of point $P_2$.
By the maximum principle, $p_\sigma$ converges locally in $C^2(\Omega\cap\hat{B})$
to a solution $p\in C^{2,\alpha}(\Omega\cap \hat{B})$ as $\sigma\rightarrow 0+$.

Next we construct a barrier function to prove the continuity of $p$ at $P_2$.
Define
$$
u=\hat{p}^{\varepsilon}+C(1-e^{-l(\theta-\frac{3\pi}{2})}),
$$
where $C>0$ and $l>0$ are specified later.
For the equation, we have
$$
Q^{\varepsilon,+}u=-\frac{Cl^2(u+\varepsilon)}{r^2}e^{-l(\theta-\frac{3\pi}{2})}<0.
$$
It is direct to see that $u\geq \hat{p}^{\varepsilon}$ on $L_{\varepsilon}$.
Choose $C$ large enough such that
$$
u\geq \sup|h| \qquad \mbox{on $\partial \hat{B}\cap \Omega$}.
$$
For the oblique derivative boundary condition along $\partial\Omega_{\sigma}\cap\Gamma_{\rm shock}^{\varepsilon}$,
we find that
\begin{itemize}
 \item $(\beta_1,\beta_2)\cdot{\boldsymbol \nu}>0$\, and\, $\mathcal{B}u>0$ for $\theta<\frac{3\pi}{2}$,

 \smallskip
 \item $(\beta_1,\beta_2)\cdot{\boldsymbol \nu}<0$\, and \,\,$\mathcal{B}u<0$ for $\theta>\frac{3\pi}{2}$,
\end{itemize}
where ${\boldsymbol \nu}$ denotes the outward normal to $\partial\Omega_{\sigma}\cap\Gamma_{\rm shock}^{\varepsilon}$.
Thus, by the comparison principle, we have
$$
\hat{p}^{\varepsilon}\leq p\leq u,
$$
which implies that $p$ is continuous at $P_2$.

Then we can follow the arguments in \cite{lieberman1986mixed,gilbarg2015elliptic} to
prove the existence of a solution away from point $P_2$ and obtain the estimates near both points $P_1$ and $P_3$,
and both boundaries $\Gamma_{\rm sonic}$ and $\Gamma_{\rm shock}^{\varepsilon}\backslash \Omega_{P_2}(d)$.
This completes the proof.
\end{proof}

Since the interior Schauder estimates can be further applied,
any solution in $C^{2,\alpha}_{\rm loc}(\Omega^{\prime})$ is actually in $C^{3}_{\rm loc}(\Omega)$.
We next establish the H\"{o}lder gradient estimate of the solution on $\Gamma_{\rm shock}^{\varepsilon}$.

\begin{lemma}\label{le3-3}
Assume that $\Gamma_{\rm shock}^{\varepsilon}$ is given by $\{(r(\theta),\theta)\}$ with
$r(\theta)\in \mathcal{R}^{\varepsilon}$ for some $\alpha_1\in(0,1)$.
Then there exists a positive constant $d_0$ such that, for every $d\leq d_0$,
any solution $p^\varepsilon\in C^1_{\rm loc}(\Omega\cup\Gamma_{\rm shock})\cap C^3_{\rm loc}(\Omega)$
of Problem {\rm \ref{pr3-1}} satisfies
\begin{equation}\label{estimate3}
\|p^\varepsilon\|_{1+\mu,\Gamma(d)\backslash \Omega_{V}(d_0)}
\leq C(\varepsilon,\alpha_1,\mu,\gamma_1,K_0,K_1)\|p^\varepsilon\|_0
\end{equation}
for any $\mu<\min\{\alpha_1,\gamma_1\}$.
\end{lemma}

\begin{proof}
Away from a neighborhood $B_{d_0}(P_2)$ of $P_2$,
we can employ Theorem $6.30$ in \cite{gilbarg2015elliptic} to obtain
\eqref{estimate3} in $\Gamma(d)\backslash \big(\cup_{i=1,2,3}B_{d_0}(P_i)\big)$.
For the estimates near $P_2$, we follow the technique used in
\cite{canic2006free}.
The main idea is that, for a given solution $p$ of the linear problem  \eqref{linearforp}--\eqref{eqcond},
we define
$$
u=\frac{p}{1+\|Dp\|_{0}},\qquad\, z=\mathcal{B}u:=\sum\limits_{i=1}^{2}\beta_i(P)D_{i}u.
$$
Taking $f(\zeta)=f_0\zeta^{\mu}$ for any $\mu<\gamma_2$, we can prove that $\pm f(\zeta)$ are barrier
functions for $z$ on $B_{d_{0}}(P_2)\cap \overline{\Omega}$ such that $|z|\leq f$.
The barrier functions lead to
$$
|D(z+f)|\leq \|(z+f)\|_{1+\gamma_2}^{1-\mu}d^{\mu-1}\leq Cd^{\mu-1} \qquad \text{for $d<d_0$},
$$
which implies that $\|u\|_{1+\mu}\leq C$.

It follows from the interpolation inequality that
$$
\|p\|_{1+\mu}\leq C\big(1+\|Dp\|_{0}\big)\leq C\big(1+\alpha\|p\|_{1,\mu}+C_{\alpha}\|p\|_{0}\big)
$$
with small $\alpha>0$, so that \eqref{estimate3} holds.
Also see \cite{canic2006free} for more details.
\end{proof}

\subsection{Regularized nonlinear fixed boundary value problem}
This subsection is devoted to the proof of the existence of solutions of the nonlinear
equation \eqref{regularizedforp} with a fixed boundary $r(\theta)\in\mathcal{R}^\varepsilon$.
We have the following lemma.

\begin{lemma}\label{nonfixed}
For each $\varepsilon\in(0,1)$, given $r^\varepsilon(\theta)\in\mathcal{R}^\varepsilon$,
there exists a solution $p^\varepsilon\in C^{2,\alpha}_{-\gamma}(\Omega)$ of
equation \eqref{regularizedforp} with conditions \eqref{eqrh}--\eqref{eqrh2}
for $r^{\varepsilon}(\theta)$ and \eqref{eqcond} such that
\begin{equation}
p_2<\overline{p}^\varepsilon\leq p^\varepsilon<p_1,\quad p^\varepsilon>r^2
\qquad\text{in $\overline{\Omega}$}.
\end{equation}
Moreover, for some $d_0>0$, $p^\varepsilon(r,\theta)$ satisfies
\begin{equation}
\|p^\varepsilon\|_{\gamma,\Gamma(d_0)\cup \Omega_{V}(d_0)}\leq K_1,
\end{equation}
where $\gamma$ and $K_1$ depend on $\varepsilon$, $\gamma_V$, and $K$, but
independent of $\alpha_1$.
\end{lemma}

\begin{proof} For simplicity, we suppress the $\varepsilon$--dependence in the proof.
Using the H\"{o}lder gradient bounds for the linear problem, we establish the existence results for
the nonlinear fixed boundary problem via the Schauder fixed point theorem.

For any function $w\in \mathcal{W}$, we define a mapping
\begin{equation} \label{eqmapping}
T:\, \mathcal{W}\subset C^{2}_{(-\gamma_1)}\rightarrow C^{2}_{(-\gamma_1)}
\end{equation}
by $Tw=p$,
where $p$ is the solution of the linear regularized fixed boundary value problem \eqref{linearforp}--\eqref{eqcond}
solved in Lemma \ref{le2-1}.
It is direct to see that $T$ maps $\mathcal{W}$ into a bounded set in $C^{2}_{(-\gamma_v)}$,
where $\gamma_{v}$ is determined in Lemma \ref{le2-1}.
Since $\gamma_{v}$ is independent of $\gamma_1$, we may take $\gamma_1=\frac{\gamma_v}{2}$
so that $T(\mathcal{W})$ is precompact in $C^{2}_{(-\gamma_1)}$.

Next, we show that $T$ maps $\mathcal{W}$ into itself.
First, $(W1)$ is satisfied by the boundary conditions and the maximum principle.
$(W2)$ is satisfied by the standard interior and boundary H\"{o}lder estimates for elliptic equations.
In order to prove that $p$ satisfies $(W3)$, it suffices to prove that there exists $K>0$ such that
\begin{equation}\label{eq314}
\sup\limits_{\sigma>0}\big(\sigma^{2-\gamma_1}\|p\|_{2,\overline{\Omega}\backslash (\Gamma(\sigma)\cup \Omega_{V}(\sigma))}\big)<K,
\end{equation}
under the condition that $\|w\|_{2}^{(-\gamma_1)}\leq K$.
Lemma  \ref{le3-3} implies that
$$
d^{2-\gamma_1}\|p\|_{2}\leq d^{1-\gamma_1+\mu}C \qquad\,\,\mbox{for all $d\leq d_0$},
$$
where $C$ depends on $K$, $\alpha_1$, and $\gamma_1$.
Moreover, by the interpolation inequality,
we can obtain
$$
d^{2-\gamma_1}\|p\|_2\leq K_V\qquad \text{for all $d<d_V$},
$$
where $\gamma_1=\frac{\gamma_V}{2}$, and $K_V$ is independent of $K$.
Therefore, we can choose sufficiently small $\hat{d}\leq \frac{\min\{d_0,d_V\}}{2}$ such that
$$
\hat{d}^{1-\gamma_1+\mu}C\leq K.
$$
For domain
$\overline{\Omega}\backslash \big(\Gamma(\sigma)\cup \Omega_{V}(\sigma)\big)$  with $\sigma>\hat{d}$,
the solution is smooth, and its $C^2$--norm bound is independent
of $K$ by the uniform H\"{o}lder estimate.
Therefore,  \eqref{eq314} is satisfied, and  parameters $K$, $K_0$, and $\alpha_0$
defining $\mathcal{W}$ have been chosen such that $T$ maps $\mathcal{W}$ into itself.

Finally, by the Schauder fixed point theorem, there exists a fixed point $p$ such that
$$
Tp=p\in C^{2}_{(-\gamma_1)}.
$$
Then $p$ is a solution as required.
\end{proof}

\subsection{Regularized nonlinear free boundary value problem}
We now prove the existence of a solution of the regularized free boundary problem \eqref{regularizedforp}, \eqref{eqrh}--\eqref{eqrh2},
and \eqref{eqcond}.
For each $r(\theta)\in \mathcal{R}^\varepsilon$,
using the solution, $p$, of the nonlinear fixed boundary problem given by Lemma 3.3,
we define the map, $J$,
on $\mathcal{R}^\varepsilon$:
\begin{equation} \label{mapping2}
\widetilde{r}=Jr:
=\begin{cases}
r_1+\int_{\theta_1}^{\theta}g(r(s),s,p(s,r(s)))\dd s \quad& \text{for $\theta\in[\frac{3\pi}{2},\theta_1)$},\\[5pt]
r_1+\int^{\theta_3}_{\theta}g(r(s),s,p(s,r(s)))\dd s \quad& \text{for $\theta\in(\theta_3,\frac{3\pi}{2}]$}.
\end{cases}
\end{equation}
First, we check that $J$ maps $\mathcal{R}^\varepsilon$ into itself.
Property (R1) follows from \eqref{mapping2}.
By the definition of $g$ and $r^2(\frac{3\pi}{2})=\overline{p}$,
property (R3) holds. In order to prove property (R2),
we need to make clear the position of the diffracted shock as a free boundary.

There are three possibilities for the position of the diffracted shock $\Gamma_{\rm shock}$:
\begin{enumerate}
\item[(i)] $r_2<r(\theta)\leq r_1$ for all $\theta \in [\theta_3,\theta_{1}]$;

\smallskip
\item[(ii)] There exists $\theta^*>0$ such that $r(\theta^*)=r_2$ for all $\theta\in [\frac{3\pi}{2}-\theta^*,\frac{3\pi}{2}+\theta^*]$,
and $r_2<r(\theta)$ for all $\theta\in[\theta_{3},\frac{3\pi}{2}-\theta^*)\cup(\frac{3\pi}{2}+\theta^*,\theta_1]$;

\smallskip
\item[(iii)] $r(\frac{3\pi}{2})=r_2$, and $r_2<r(\theta)\leq r_1$ for all $\theta\in[\theta_{3},\theta_{1}]\backslash\{\frac{3\pi}{2}\}$,
where $r_i=\sqrt{p_i}$,  and $\theta_i$ are the $\theta-$coordinates of point $P_i$,  $i=1,2,3$.
\end{enumerate}
Let
$$
d=\text{dist}\{C_2,\Gamma_{\rm shock}\}=|OP_2|-r_2.
$$

\begin{proposition}\label{le4.1}
Let $(p,r)$ be the solution of the regularized boundary problem  \eqref{regularizedforp}, \eqref{eqrh}--\eqref{eqrh2},
and \eqref{eqcond}.
Then $d\geq\delta>0$ for some constant $\delta>0$ depending on the Riemann initial data, when $\varepsilon>0$ is sufficiently small.
This means that cases {\rm (ii)}--{\rm (iii)} do not occur, so that the
diffracted shock does not meet the sonic circle $\{r=r_2\}$.
\end{proposition}

\begin{proof}
We first prove (ii) via the method of contradiction.
For $\theta_0\in [\frac{3\pi}{2}-\theta^*,\frac{3\pi}{2}+\theta^*]$, let  $\mathcal{N}$ be a small interior neighborhood  of
the point $(r_2, \theta_0)\in \Gamma_{\rm shock}$.
We now prove that the optimal regularity of $p$ in $\mathcal{N}$ near $\Gamma_{\rm shock}$ is $C^{1/2}$.

\medskip
1. We introduce the barrier function
$$
w(r,\theta)=p_2+A_1(r_2-r)^{\frac{1}{2}}-B_1(r_2-r)^{\beta}+D_1(\theta-\theta_0)^2,$$
where $A_1,B_1,D_1>0$, and $\beta\in(\frac{1}{2},1)$ will be specified later. Let
\begin{equation}
Qw:=(p-r^2+\varepsilon)w_{rr}+\frac{p+\varepsilon}{r^2}w_{\theta\theta}
+\frac{p+\varepsilon}{r}w_r+\frac{1}{p}(rw_r)^2-2rw_r
\end{equation}
and let $\hat{Q}w$ be obtained by replacing the coefficient of $w_{rr}$ in $Qw$ by $w-r^2$, that is,
$$
\hat{Q}w:=Qw+(w-p-\epsilon)w_{rr}.
$$
A direct calculation yields
\begin{equation*}
\begin{split}
\hat{Q}w=&\Big((\frac{1}{4}-\beta^2)A_1B_1(r_2-r)^{\beta-\frac{3}{2}}+O_1\Big)-\frac{A_1^2}{4p}(p-r^2)(r_2-r)^{-1}\\
&+\left(B_1^2\beta(2\beta-1)(r_2-r)^{2\beta-2}+O_2\right)+\Big(-\frac{A_1D_1}{4}(r_2-r)^{-\frac{3}{2}}(\theta-\theta_0)^2
+O_3\Big),
\end{split}
\end{equation*}
where
\begin{align*}
O_1&=\frac{p-r^2}{p}A_1B_1\beta(r_2-r)^{\beta-\frac{3}{2}}+\frac{A_1(2r^2-p)}{2r}(r_2-r)^{-\frac{1}{2}}+\frac{p-2r^2}{r}B_1\beta(r_2-r)^{\beta-1}\\
&\quad-(p_2-r^2)\Big(\frac{A_1}{4}(r_2-r)^{-\frac{3}{2}}+B_1\beta(\beta-1)(r_2-r)^{\beta-2}\Big)+\frac{2pD_1}{r^2},\\
O_2&=-\frac{B_1^2\beta^2}{p}(p-r^2)(r_2-r)^{2\beta-2},\\
O_3&=B_1D_1\beta(1-\beta)(r_2-r)^{\beta-2}(\theta-\theta_0)^2.
\end{align*}

Notice that there exists $\alpha\in(0,\frac{1}{2})$ such that $p-r^2\leq (r_2-r)^{\alpha}$. Then
$$
\left|-\frac{A_1^2}{4p}(p-r^2)(r_2-r)^{-1}\right|\leq C(p_1,p_2)A_1^2(r_2-r)^{\alpha-1}.
$$

Choose $\beta$ such that $\beta-\frac{3}{2}<\alpha-1$, {\it i.e.}, $\alpha>\beta-\frac{1}{2}$ so that
$$
(\beta^2-\frac{1}{4})A_1B_1(r_2-r)^{\beta-\frac{3}{2}}>3C(p_1,p_2)A_1^2(r_2-r)^{\alpha-1},
$$
which implies
\begin{equation}
\label{eq3.2}
B_1>A_1C(p_1,p_2,\beta)(r_2-r)^{\alpha-\beta+\frac{1}{2}}.
\end{equation}
On the other hand, if $r_2-r$ is small enough,
$$3B_1^2\beta(2\beta-1)(r_2-r)^{2\beta-2}<(\beta^2-\frac{1}{4})A_1B_1(r_2-r)^{\beta-\frac{3}{2}},$$
which implies
\begin{equation}
\label{eq3.3}
A_1>\frac{3\beta(2\beta-1)B_1}{\beta^2-\frac{1}{4}}(r_2-r)^{\beta-\frac{1}{2}}.
\end{equation}
We can choose $A_1$ and $D_1$ such that
\begin{equation}
\label{eq3.4}
w(r,\theta)>p_2+\frac{A_1}{2}(r_2-r)^{\frac{1}{2}}+D_1(\theta-\theta_0)^2>p \qquad \text{on $\partial\mathcal{N}\setminus\{r=r_2\}$}.
\end{equation}
Moreover, we see that
\begin{equation}
\label{eq3.4b}
w(r,\theta)=p_2+D_1(\theta-\theta_0)^2\ge p \qquad \text{on $\partial\mathcal{N}\cap \{r=r_2\}$},
\end{equation}
which implies that
\begin{equation}
\label{eq3.4c}
w(r,\theta)\ge p \qquad \text{on $\partial\mathcal{N}$}.
\end{equation}

On the other hand, we take $B_1$ sufficiently small such that \eqref{eq3.2}--\eqref{eq3.3} hold. Finally we have
$$
\hat{Q}w<0 \qquad \text{in $\mathcal{N}$}.
$$
Moreover, for sufficiently small $r_2-r$,
$$
\partial_{rr}w(r,\theta)=-\frac{1}{4}A_1(r_2-r)^{-\frac{3}{2}}+B_1\beta(1-\beta)(r_2-r)^{\beta-2}<0.
$$

Now assume that there exists a non-empty open subset $\mathcal{N}_1=\{(r,\theta)\in\mathcal{N},\,p>w\}\subset\mathcal{N}$.
By the continuity of function $p-w$, there exist a maximum point $P_*\in \mathcal{N}_1$ and
a small neighborhood $\hat{\mathcal{N}}_1\subset \mathcal{N}_1$ such that
$$
p-w>\epsilon \qquad \mbox{on $\hat{\mathcal{N}}_1$},
$$
when $\varepsilon>0$ is sufficiently small.
Then
\begin{equation}\label{3.5a}
Q(p-w)=Qp-\hat{Q}w +(\hat{Q}w-Qw)=-\hat{Q}w + (w-p-\varepsilon)w_{rr}\geq -\hat{Q}w>0 \quad \text{in $\hat{\mathcal{N}}_1$};
\end{equation}
in particular,
\begin{equation}\label{3.5b}
Q(p-w)(P_*)>0 \qquad \text{in $\hat{\mathcal{N}}_1$}.
\end{equation}
On the other hand, at the maximum point $P_*\in \hat{\mathcal{N}}_1$, the structure of operator $Q$ leads to
$$
Q(p-w)(P_*) \leq 0,
$$
which is a contradiction to \eqref{3.5b}, so that the open subset $\mathcal{N}_1\subset \mathcal{N}$ must be empty.
Thus, $p\leq w$ in $\mathcal{N}$.

\medskip
2. Next we take
$$
v(r,\theta)=p_2+A_2(r_2-r)^{\frac{1}{2}}+B_2(r_2-r)^{\beta}-D_2(\theta-\theta_0)^2,
$$
where $A_2,B_2,D_2>0$ and $\frac{1}{2}<\beta<1$. Then
\begin{equation*}
\begin{split}
\hat{Q}v=&\Big((\beta^2-\frac{1}{4})A_2B_2(r_2-r)^{\beta-\frac{3}{2}}+\overline{O}_1\Big)-\frac{A_2^2}{4p}(p-r^2)(r_2-r)^{-1}\\
&+\left(B_2^2\beta(2\beta-1)(r_2-r)^{2\beta-2}+\overline{O}_2\right)+
\Big(\frac{A_1D_1}{4}(r_2-r)^{-\frac{3}{2}}(\theta-\theta_0)^2
+\overline{O}_3\Big),
\end{split}
\end{equation*}
where
\begin{align*}
\overline{O}_1&=-\frac{p-r^2}{p}A_2B_2\beta(r_2-r)^{\beta-\frac{3}{2}}
+\frac{A_2(2r^2-p)}{2r}(r_2-r)^{-\frac{1}{2}}+\frac{2r^2-p}{r}B_2\beta(r_2-r)^{\beta-1}\\
&\quad-(p_2-r^2)\Big(\frac{A_2}{4}(r_2-r)^{-\frac{3}{2}}+B_2\beta(1-\beta)(r_2-r)^{\beta-2}\Big)-\frac{2pD_2}{r^2},\\
\overline{O}_2&=-\frac{B_2^2\beta^2}{p}(p-r^2)(r_2-r)^{2\beta-2},\\
\overline{O}_3&=-B_2D_2\beta(1-\beta)(r_2-r)^{\beta-2}(\theta-\theta_0)^2.
\end{align*}
Similarly,
there exists $\alpha\in(0,\frac{1}{2})$ such that $p-r^2\leq (r_2-r)^{\alpha}$. Then
$$
\left|-\frac{A_2^2}{4p}(p-r^2)(r_2-r)^{-1}\right|\leq C(p_1,p_2)A_2^2(r_2-r)^{\alpha-1}.
$$

Now we choose $\beta$ such that $\beta-\frac{3}{2}<\alpha-1$, {\it i.e.}, $\alpha>\beta-\frac{1}{2}$. Then
$$
(\beta^2-\frac{1}{4})A_2B_2(r_2-r)^{\beta-\frac{3}{2}}>3C(p_1,p_2)A_2^2(r_2-r)^{\alpha-1},
$$
which implies
\begin{equation}\label{eq4.4}
B_2>A_2C(p_1,p_2,\beta)(r_2-r)^{\alpha-\beta+\frac{1}{2}}.
\end{equation}
On the other hand, if $r_2-r$ is small enough,
$$
3B_2^2\beta(2\beta-1)(r_2-r)^{2\beta-2}<(\beta^2-\frac{1}{4})A_2B_2(r_2-r)^{\beta-\frac{3}{2}},
$$
which implies
\begin{equation}
\label{eq4.5}
A_2>\frac{3\beta(2\beta-1)B_2}{\beta^2-\frac{1}{4}}(r_2-r)^{\beta-\frac{1}{2}}.
\end{equation}
Let $D_2$ be large enough such that $p>v$ for some
$\theta=\theta_a,\theta_b$ so that $\theta_0\in (\theta_a, \theta_b)$. We choose $\widetilde{r}<r_2$ such that
\begin{align*}
p&>p_2+2A_2(r_2-\widetilde{r})^{\frac{1}{2}}-D_2(\theta-\theta_0)^2\\
&\geq p_2+A_2(r_2-\widetilde{r})^{\frac{1}{2}}+B_2(r_2-\widetilde{r})^{\beta}-D_2(\theta-\theta_0)^2=v(\tilde{r}, \theta),
\end{align*}
where the second inequality holds, provided that $\frac{B_2}{A_2}\leq (r_2-\widetilde{r})^{\frac{1}{2}-\beta}$.
Moreover,
$$
(p-v)|_{\{r=r_2\}}=D_2(\theta-\theta_0)^2\ge 0     \qquad\text{for $\theta\in (\theta_a, \theta_b)$},
$$
which implies that
$$
(p-v)|_{\partial\tilde{\mathcal{N}}}\ge 0,
$$
where $\tilde{\mathcal{N}}=\{r\in (\tilde{r}, r_2), \, \theta\in (\theta_a, \theta_b)\}$.

On the other hand,
we take $B_2$ sufficiently small such that \eqref{eq4.4}--\eqref{eq4.5} hold. Then we  derive
$$
\hat{Q}v>0  \qquad \text{in $\tilde{\mathcal{N}}$}.
$$
Moreover, in $\tilde{\mathcal{N}}$,
$$
\partial_{rr}v(r,\theta)=-\frac{1}{4}A_1(r_2-r)^{-\frac{3}{2}}+B_1\beta(1-\beta)(r_2-r)^{\beta-2}<0.
$$

Now assume that there exists a non-empty open subset $\mathcal{N}_2=\{(r,\theta)\in\tilde{\mathcal{N}},p<v\}\subset \tilde{\mathcal{N}}$.
By the continuity of function $p-v$, there exist a minimum point $P_{**}\in \mathcal{N}_2$ and
a small neighborhood $\hat{\mathcal{N}}_2\subset \mathcal{N}_2$ such that
$$
p-v<-\epsilon \qquad \mbox{on $\hat{\mathcal{N}}_2$}
$$
when $\varepsilon>0$ is sufficiently small.
Then
\begin{equation}\label{3.6a}
Q(p-v)=Qp-\hat{Q}v +(\hat{Q}v-Qv)=-\hat{Q}v + (v-p-\varepsilon)v_{rr}\leq  -\hat{Q}v<0 \quad \text{in $\hat{\mathcal{N}}_2$};
\end{equation}
in particular,
\begin{equation}\label{3.6b}
Q(p-v)(P_{**})<0 \qquad \text{in $\hat{\mathcal{N}}_2$}.
\end{equation}
On the other hand, at the minimum point $P_{**}\in \hat{\mathcal{N}}_2$, the structure of operator $Q$ leads to
$$
Q(p-v)(P_{**}) \geq 0,
$$
which is a contradiction to \eqref{3.6b}, so that the open subset $\mathcal{N}_2\subset \tilde{\mathcal{N}}$ must be empty.
Thus, $p\geq v$ in $\tilde{\mathcal{N}}$.

\medskip
3. Combining Steps 1--2,  we conclude that
$$
a(r_2-r)^{\frac{1}{2}}<p-p_2<A(r_2-r)^{\frac{1}{2}} \qquad \text{in $\mathcal{N}$}
$$
for some constants $a,A>0$, so that the optimal regularity of $p$  near the sonic circle $C_2$ is  $C^{\frac{1}{2}}$.

\medskip
4. Next we introduce the coordinates $(x,y)=(r_2-r,\theta)$ and denote $\varphi:=p-p_2$. We scale $\varphi$
in $\mathcal{N}$ such that
$$u(S,T)=\frac{1}{S^{1/5}}\varphi(S^{-\frac{12}{5}},y_0+S^{-\frac{14}{5}}T) \qquad \text{for}\,\,
(S^{-\frac{12}{5}},y_0+S^{-\frac{14}{5}}T)\in \mathcal{N}.$$
From the optimal regularity, it follows that
$$0<a\leq S^{\frac{7}{5}}u\leq A$$
for some constants $a,A>0$. From \eqref{eqforp2}, we obtain the governing equation for $u$ in the $(S,T)$--coordinates
\begin{equation}
\label{eq4.7}
Qu=a_{11}u_{SS}+a_{12}u_{ST}+a_{22}u_{TT}+b_{1}u_S+b_{2}u_T+c_{1}u+\text{h.o.t}=0,
\end{equation}
where
\begin{align*}
&a_{11}=S^{\frac{7}{5}}u+2r_2S^{-\frac{6}{5}}-S^{-\frac{18}{5}},\\
&a_{12}=\frac{28T}{5S}\big(S^{\frac{7}{5}}u+2r_2S^{-\frac{6}{5}}-S^{-\frac{18}{5}}\big),\\
&a_{22}=\frac{144}{25}+\frac{196T^2}{S^2}\big(S^{\frac{7}{5}}u+2r_2S^{-\frac{6}{5}}-S^{-\frac{18}{5}}\big),\\
&b_{1}=\frac{19}{5S}\big(S^{\frac{7}{5}}u+2r_2S^{-\frac{6}{5}}-S^{-\frac{18}{5}}\big)+
\frac{12M}{5S^{\frac{11}{5}}}\Big(\frac{S^{\frac{1}{5}}u+p_2}{M^2}-2\Big)=:\hat{b}_{1}S^{-1},\\
&b_{2}=\frac{294T}{25S^{2}}\big(S^{\frac{7}{5}}u+2r_2S^{-\frac{6}{5}}-S^{-\frac{18}{5}}\big)+
\frac{168MT}{25S^{\frac{16}{5}}}\Big(\frac{S^{\frac{1}{5}}u+p_2}{M^2}-2\Big)=:\hat{b}_{2}S^{-2}T,\\
&c_{1}=\frac{13}{25S^{2}}\big(S^{\frac{7}{5}}u+2r_2S^{-\frac{6}{5}}-S^{-\frac{18}{5}}\big)
+\frac{12M}{25S^{\frac{16}{5}}}\Big(\frac{S^{\frac{1}{5}}u+p_2}{M^2}-2\Big)=:\hat{c}_{1}S^{-2},\\
&\text{h.o.t}=\frac{M^2S^{\frac{2}{5}}}{S^{\frac{1}{5}}u+p_0}\Big(S(u_S)^2
+\frac{196}{25}\frac{T^2}{S}(u_T)^2
+\frac{1}{25}\frac{u^2}{S}+\frac{28}{5}Tu_Su_T+\frac{2}{5}uu_S+\frac{28}{25}\frac{T}{S}uu_T\Big),
\end{align*}
with $M:=r_0-S^{-\frac{12}{5}}$.
According to the optimal regularity, we have
$$
0< C^{-1}\leq a_{11},a_{12},a_{22},\hat{b}_{1},\hat{b}_{2},\hat{c}_1\leq C, \qquad
\text{h.o.t.}\rightarrow 0,
$$
if $S^{-1}$ and $T$ are sufficiently small. Moreover, the eigenvalues of \eqref{eq4.7} are positive and bounded
so that \eqref{eq4.7} is uniformly elliptic for $u$ in the $(S,T)-$coordinates.

Let $x_0^{-1}<S<x_0^{-2}$ with $x_0$ small enough. Then, using Theorem $8.20$ in \cite{gilbarg2015elliptic},
we conclude
\begin{align*}
ax_0^{\frac{7}{5}}&\leq u(x_0^{-1},0)\leq \sup_{x_0^{-1}\leq S \leq x_0^{-2}}u(S,T)\\
&\leq C\inf_{x_0^{-1} \leq S \leq x_0^{-2}}u(S,T)\leq Cu(x_0^{-2},0)\leq CAx_{0}^{\frac{14}{5}},
\end{align*}
where $C$ is independent of $x_0$.
This implies that $x_0^{-\frac{7}{5}}\leq K$ for a bounded constant $K>0$, which is a contradiction if
$x_0$ is sufficiently small.
Therefore, case (ii) can not occur.

It can be proved the impossibility of case (iii) similarly, so we omit the proof here. This completes the proof.
\end{proof}

In order to use the Schauder fixed point theorem, we need further to prove that map $J$
is compact and continuous on $\mathcal{R}^\varepsilon$.
Evaluating $g(r,\theta,p)$, we obtain the bound:
$$
\|g\|_{\gamma_V/2}\leq C(K_1),
$$
so that
$$
\|\widetilde{r}\|_{1+\gamma_V/2}\leq C(K_1),
$$
where $\gamma_{V}$ is independent of $\alpha_1$,
the H\"{o}lder exponent of space $\mathcal{R}^\varepsilon$.
Thus, $J(\mathcal{R}^\varepsilon)\subset C^{1+\frac{\gamma_V}{2}}$, and
$J(\mathcal{R}^\varepsilon)\subset \mathcal{R}^\varepsilon$ if $\alpha_1\leq \frac{\gamma_V}{2}$.
We take $\alpha_1=\frac{\gamma_V}{3}$ to guarantee that $J$ is compact.
Furthermore, assume that
$r_m,r\in \mathcal{R}^\varepsilon$ and $r_m\rightarrow r$ as $m\rightarrow \infty$.
Assume that $p_m$ is the solution of the nonlinear fixed boundary problem with shock $\Gamma_{\rm shock}$
defined by $r_m$ for each $m$. By the standard argument ({\it cf.} \cite{canic2006free}),
we obtain that $p_m\rightarrow p$, which solves
the problem for $r$.
Then
$$
g(r_m(\theta),\theta,p(r_m(\theta),\theta))\rightarrow g(r(\theta),\theta,p(r(\theta),\theta))
\qquad \mbox{as $m\rightarrow \infty$},
$$
which implies that $Jr_m\rightarrow Jr$ as $m\rightarrow \infty$.
Therefore, the Schauder fixed point theorem implies that $J$ has a fixed point
$r^\varepsilon\in C^{1+\frac{\gamma_V}{3}}([\theta_3,\theta_1])$.

Together with the corresponding solution $p^\varepsilon$,
we conclude  the existence of a solution
$(p^\varepsilon,r^{\varepsilon})\in
C^{2,\alpha}_{(-\gamma)}(\Omega^\varepsilon)\times C^{1,\alpha_1}([\theta_3,\theta_1])$
of the regularized free boundary problem \eqref{regularizedforp}, \eqref{eqrh}--\eqref{eqrh2},
and \eqref{eqcond}.

\subsection{Existence of solutions of Problem \ref{problem}}

In this section, we prove that the limit of $(p^\varepsilon,r^\varepsilon)$ as $\varepsilon\rightarrow 0$ is
a solution of Problem \ref{problem}.

\begin{lemma}
There exists a positive function $\phi$, independent of $\varepsilon$, such that
\begin{equation}
p^\varepsilon-(\xi^2+\eta^2)\geq \phi \qquad\, \text{in $\overline{\Omega}\backslash \Gamma_{\rm sonic}$},
\end{equation}
and $\phi\rightarrow 0$ as $\emph{dist}((\xi,\eta),\Gamma_{\rm sonic})\rightarrow 0$.
\end{lemma}

\begin{proof}
For $X_0=(\xi_0,\eta_0)\in \Omega$ and $0<R<1$, denote
$$
\zeta(X)=1-\frac{(\xi-\xi_0)^2+(\eta-\eta_0)^2}{R^2} \qquad\, \text{for $B_R(X_0)\cap \Gamma_{\rm sonic}=\emptyset$}.
$$
We define
$$
\phi=\delta_0(\zeta(X))^\tau,
$$
where $\delta_0$ and $\tau$ are two positive constants.
We can obtain a local uniform lower-barrier
$$
p^\varepsilon-(\xi^2+\eta^2)\geq \phi=\delta_0(\zeta(X))^\tau \qquad\, \text{in $B_{\frac{3}{4}}(X_0)\cap \overline{\Omega}$},
$$
where  $\delta_0$ and $\tau$ are  independent of $\varepsilon$.
Moreover,
$$
\delta_0\rightarrow 0\qquad\, \text{as $\text{dist}((\xi,\eta),\Gamma_{\rm sonic})\rightarrow 0$},
$$
so does $\phi$. See \cite{canic2006free,chen2014shock} for more details.
\end{proof}

The uniform lower bound of $p^\varepsilon-(\xi^2+\eta^2)$, independent of $\varepsilon$,
implies that the governing equation is locally uniform elliptic,
independent of $\varepsilon$.
Thus, we can apply the standard local compactness arguments to obtain the limit, $p$,
locally in the interior of the domain.

\begin{theorem}
There exist functions $r(\theta)\in C^1([\theta_3,\theta_1])$ and $p\in C^{2,\alpha}_{\rm loc}(\Omega)\cap C(\overline{\Omega})$
such that
$$
r^\varepsilon\rightarrow r \quad \text{in $C([\theta_3,\theta_1])$}, \qquad\,\,
p^\varepsilon\rightarrow p \quad \text{in $C^{2,\alpha}_{\rm loc}(\Omega)$},
$$
and $(p,r)$ is a solution of the free boundary value problem, Problem {\rm \ref{problem}}.
\end{theorem}

\begin{proof}
We have obtained the estimate:
$$
\|r^\varepsilon(\theta)\|_{C^{1}([\theta_3,\theta_1])}\leq C,
$$
where $C$ is independent of $\varepsilon$.
Then, by the Arzela-Ascoli theorem, there exists a subsequence converging uniformly to a function
$r(\theta)$ in $C^{\alpha}([\theta_3,\theta_1])$ as $\varepsilon \rightarrow 0$, for any $\alpha\in (0,1)$.
By the local ellipticity and the standard
interior Schauder estimate, there exists a function $p\in C^{2,\alpha}_{\rm loc}(\Omega)$
such that $p^{\varepsilon}\rightarrow p$ in any
compact subset, contained by $\overline{\Omega}\backslash (\Gamma_{\rm sonic}\cup\Gamma_{\rm shock})$,
satisfying that $Qp=0$.

Since the shock does not meet the sonic circle of state $p_2$,
$p^\varepsilon>p_2$.
Thus, we have the uniform ellipticity,
and the uniform negativity of $\beta\cdot \boldsymbol{\nu}$ locally.
Thus, we can pass the limit to obtain that $p\in C^{1,\alpha}$,
$$
\mathcal{B}p=0\qquad\, \text{on $\Gamma_{\rm shock}$},
$$
and $r^{\prime}(\theta)=g(r(\theta),\theta)$.
Therefore, the limiting vector function $(p,r)$ is a global solution of
Problem \ref{problem}.
\end{proof}

The property of $\Gamma_{\rm shock}$ is stated as follows:

\begin{proposition}
For the free boundary $\Gamma_{\rm shock}=\{(\xi,\eta(\xi))\, :\,\xi_3<\xi<\xi_1\}$ with $\xi_3$ and $\xi_1$
as the $\xi$--coordinates of points $P_3$ and $P_1$ respectively,
$$
\eta(\xi)\in C^2(\xi_3,\xi_1),
$$
and $\eta(\xi)$ is strictly convex for $\xi\in (\xi_3,\xi_1)$.
\end{proposition}

\begin{proof}
Define
$$
F(\xi,\eta)=\xi^2+\eta^2-r^2(\theta(\xi,\eta))=0\qquad \text{on}\,\, \Gamma_{\rm shock}.
$$
Then
\begin{equation}
\label{eqF}
F_\eta|_{\xi=0}=(2\eta-2rr^\prime\theta_{\eta})|_{\xi=0}=2\eta(0)\neq 0.
\end{equation}
By the implicit function theorem, there exists $\eta=\eta(\xi)$ such that \eqref{eqF} holds locally on $\Gamma_{\rm shock}$
near $\xi=0$. Hence, there exists $\overline{\xi}>0$ such that $(\xi,\eta(\xi))\in \Gamma_{\rm shock}$
for $0<|\xi|\leq \overline{\xi}$.

Recall that
$$
\eta^\prime(\xi)=f(\xi,\eta(\xi),\overline{p})=
\frac{\xi\eta+\sqrt{\overline{p}\,(\xi^2+\eta^2-\overline{p})}}{\xi^2-\overline{p}}.
$$
Then
$$
\eta^{\prime\prime}(\xi)=f_\xi+\eta^\prime f_\eta+f_{\overline{p}}\,\overline{p}^\prime.
$$
Notice that
\begin{align*}
f_\xi=&\frac{\eta}{\xi^2-\overline{p}}+\frac{\xi\,\overline{p}}
{(\xi^2-\overline{p})\sqrt{\overline{p}\,(\xi^2+\eta^2-\overline{p})}}
-\frac{2\xi\big(\xi\eta+\sqrt{\overline{p}\,(\xi^2+\eta^2-\overline{p})}\big)}{(\xi^2-\overline{p})^2},\\
f_\eta=&\frac{\xi}{\xi^2-\overline{p}}+\frac{\eta\,\overline{p}}
{(\xi^2-\overline{p})\sqrt{\overline{p}\,(\xi^2+\eta^2-\overline{p})}}.
\end{align*}
Then
$f_\xi+\eta^\prime f_\eta=0$. Thus, the sign of $\eta^{\prime\prime}$ is determined by the signs of
$f_{\overline{p}}$ and $\overline{p}^\prime$.
Since
$$
\frac{\partial f}{\partial \overline{p}}=\frac{\big(\eta\sqrt{\overline{p}}+\xi\sqrt{\xi^2+\eta^2-\overline{p}}\big)^2}
{2(\xi^2-\overline{p})^2\sqrt{\overline{p}\,(\xi^2+\eta^2-\overline{p})}},
$$
we see that, for each $\xi\in (0,\xi_1)$, $\overline{p}^{\prime}>0$ and $\frac{\partial{f}}{\partial \overline{p}}>0$.
Therefore,
$$
\eta^{\prime\prime}(\xi)\geq 0 \qquad \mbox{for $\xi\in[0,\theta_1)$}.
$$
Similarly, it can be proved that
$$
\eta^{\prime\prime}(\xi)\geq 0 \qquad \mbox{for $\xi\in(\theta_3,0]$}.
$$
This implies that the shock curve $\eta(\xi)$ is strictly convex in the self-similar coordinates.
\end{proof}

\section{Optimal Regularity near the Sonic Boundary}\label{sec4}

In this section, we first establish the Lipschitz continuity for the solution near the degenerate
sonic boundary $\Gamma_{\rm sonic}$.

\begin{lemma}
The solution, $p(\xi, \eta)$ of Problem {\rm \ref{problem}} is Lipschitz continuous up to the sonic boundary $\Gamma_{\rm sonic}$.
\end{lemma}

\begin{proof}
Since $p\leq p_1$ in $\Omega$, it follows that
$$
p-\xi^2-\eta^2<p_1-\xi^2-\eta^2.
$$
On the other hand,
$p-\xi^2-\eta^2>\xi^2+\eta^2-p_1$ in $\Omega$.
Then
\begin{align*}
|p-p_1|&\leq|p-\xi^2-\eta^2|+|p_1-\xi^2-\eta^2|\\
&\leq 2\,|p_1-\xi^2-\eta^2|\leq 4\sqrt{p_1}\,\big|\sqrt{p_1}-\sqrt{\xi^2+\eta^2}\big|,
\end{align*}
which implies that $p$ is Lipschitz continuous up to the degenerate boundary $\Gamma_{\rm sonic}$.
\end{proof}

Next, we want to show that the Lipschitz continuity is the optimal regularity for $p$ across the sonic boundary $\Gamma_{\rm sonic}$,
and at the intersection points $P_1$ and $P_3$.
Since the problem is symmetric, we consider only the right-half sonic circle for convenience.

For $\epsilon\in(0,\frac{r_1}{4})$, we denote the $\epsilon$--neighborhood of the sonic
boundary $\Gamma_{\rm sonic}$ within $\Omega$ by
$$
\Omega_{\epsilon}:=\Omega\cap\{(r,\theta)\,:\, 0<r_{1}-r<\epsilon, \, \theta_1<\theta<\frac{\pi}{2}\},
$$
where we take $\theta_1\in (-\frac{\pi}{2},0)$ if $P_1$ is below the $\xi$--axis,

In $\Omega_{\epsilon}$, we introduce the
new coordinates:
\begin{equation}
(x,y):=(r_1-r,\theta-\theta_1).
\end{equation}
Then
$$
\Gamma_{\rm sonic}:=\{(0,y)\,:\,0<y<\frac{\pi}{2}-\theta_1\},\qquad\,\, P_1=(0,0).
$$
We can take $P_1$ as an interior point of $\Gamma_{\rm sonic}^{\rm ext}$,
which is obtained by reflecting $\Gamma_{\rm sonic}$ with respect to $y=0$.
Let
$$
Q^{+}_{r,R}:=\{(x,y)\,:\,x\in(0,r), |y| <R\} \qquad\,\, \text{with}\,\, R=\frac{\pi}{2}-\theta_1.
$$
Let $\varphi=p_1-p$. Then
\begin{equation}\label{eqbound}
\varphi>0\,\,\,\,  \text{in $Q_{r,R}^{+}$}, \qquad\,\, \varphi=0\,\,\,\, \text{on $\partial Q_{r,R}^{+}\cap\{x=0\}$}.
\end{equation}

It follows from \eqref{eqforp2} that $\varphi$ satisfies
\begin{equation}\label{eqforvarphi}
\mathcal{L}\varphi:=(2r_1x-\varphi+O_1)\varphi_{xx}+(1+O_2)\varphi_{yy}+(r_1+O_3)\varphi_{x}
-(1+O_4)(\varphi_{x})^2=0
\end{equation}
in $Q_{r,R}^{+}$, where
\begin{equation}\label{eqhighorder}
\begin{split}
&O_1(x,\varphi)=-x^{2},    \quad
O_2(x,\varphi)=\frac{x(2r_{1}-x)-\varphi}{(r_{1}-x)^{^{2}}},
\\[2mm]
& O_3(x,\varphi)=\frac{2x^{2}-3r_{1}x+\varphi}{r_{1}-x},
\quad O_4(x,\varphi)=\frac{x^{2}-2r_{1}x+\varphi}{r_{1}^{2}-\varphi}.
\end{split}
\end{equation}

\begin{lemma}
There exist $\varepsilon>0$ and $k>0$ depending on the Riemann initial data such that, for any solution $p$,
\begin{equation}\label{eq3-19}
0\leq\varphi \leq (2r_1-k)x \qquad\, \text{for}\,\,x\in (0,\varepsilon).
\end{equation}
\end{lemma}

\begin{proof}
We give the outline of the proof, which is similar to that for Lemma $9.6.4$ in Chen-Feldman \cite{chen2018mathematics}.
We first define a smooth
approximation to $(\xi,\eta)\rightarrow \text{dist}((\xi,\eta),\Gamma_{\rm sonic})$,
denoted  by $g(\xi,\eta)$, and then consider the function:
$p-r^2-\lambda g(\xi,\eta)$, where $\lambda>0$ will be specified later.
According to the ellipticity
principle, it can be proved that $p-r^2-\lambda g(\xi,\eta)$ cannot attain a minimum in the interior of $\Omega$.

Next, we turn to the shock boundary.
Suppose that $p-r^2-\lambda g(\xi,\eta)$ achieves its minimum at $P_{\rm min}\in\Gamma_{\rm shock}$.
Then
$$
(p-r^2)(P_{\rm min})\leq \lambda g(P_{\rm min})\leq \lambda C\qquad \text{for some}\,\, C>0.
$$
By the boundary condition \eqref{eqrh}--\eqref{eqrh2}, we can choose $\lambda$ sufficiently small such that
$$
\big((\beta_1,\beta_2)\cdot \nabla(p-r^2-\lambda g(\xi,\eta))\big)(P_{\rm min})<0,
$$
which contradicts the Hopf maximum principle.
Thus, $p-r^2-\lambda g(\xi,\eta)$ must attain its minimum on $\Gamma_{\rm sonic}$,
which implies that
\begin{equation}
\label{eqellipticity}
p-r^2\geq \lambda\, \text{dist}((\xi,\eta),\Gamma_{\rm sonic}) \qquad \text{in $\overline{\Omega}$}.
\end{equation}
It is clear that $\varphi\geq 0$.
Combining with \eqref{eqellipticity}, we can derive
\begin{equation*}
\lambda x\leq \lambda\, \text{dist}((\xi,\eta),\Gamma_{\rm sonic})\leq p-(r_1-x)^2=-\varphi+2r_1x-x^2.
\end{equation*}
This lemma can be proved by taking $k=\frac{\lambda}{2}$.
\end{proof}

According to  \eqref{eq3-19},
we derive  that $O_{i}(x,\varphi), i=1,\cdots,4$, are continuously differentiable and satisfy
\begin{equation}
\frac{|O_{1}(x,y)|}{x^2}+\frac{|O_{k}(x,y)|}{x}+\frac{|DO_{1}(x,y)|}{x}+|DO_{k}(x,y)|\leq N
\qquad\, \text{for $k=2,3,4$},
\end{equation}
for some $N>0$ depending on the Riemann initial data.
Then the leading terms of equation \eqref{eqforvarphi} form the following equation:
\begin{equation}\label{eqforvarphi2}
(2r_1x-\varphi)\varphi_{xx}+\varphi_{yy}+r_1\varphi_{x}-\varphi_{x}^2=0,
\end{equation}
which  is uniformly
elliptic in every subdomain $\{x>\delta\}$ with $\delta>0$.

\begin{lemma} \label{lem5-2}
Let $\varphi\in C^{2}(Q_{r,R}^{+})\cap C(\overline{Q_{r,R}^{+}})$ be the solution of equation \eqref{eqforvarphi} with condition
\eqref{eqbound}. Then, for any $\alpha\in (0,1)$ and $|y|<\frac{R}{2}$,
\begin{equation}\label{eq4.9}
\varphi\in C^{1,\alpha}(\overline{Q_{r/2,R/2}^{+}}), \qquad\, \varphi_{x}(0,y)=r_1, \qquad \varphi_{y}(0,y)=0.
\end{equation}
\end{lemma}

\begin{proof}
The proof is similar to that in Bae-Chen-Feldman \cite{bae2009regularity}, and we only list the major procedure and
the points of difference here.

\smallskip
1. By constructing barrier functions and the maximum principle for strictly elliptic equations,
we can prove that
$\varphi=p_1-p$ has a positive lower bound, {\it i.e.},
there exist $\widehat{r}>0$ and $\mu>0$, depending on the Riemann initial data and
$\inf_{Q_{\widehat{r},R}^{+}\cap\{x>\frac{\widehat{r}}{2}\}}\varphi$,
such that, for all $r\in(0,\frac{\widehat{r}}{2}]$,
$$
\varphi\geq \mu r_1x\qquad\,\, \text{in $Q_{r,\frac{15R}{16}}^{+}$}.
$$

2. We can now obtain more precise estimates for $\varphi$ near the sonic boundary:
$$
|\varphi(x,y)-r_1x|\leq Cx^{1+\alpha}\qquad \text{in $Q^{+}_{\widehat{r},\frac{7R}{8}}$}
$$
for any $\alpha\in (0,1)$ and some constant $C$ depending on $\widehat{r}, R, \alpha$, and the Riemann initial data.

To achieve this, we denote $W(x,y):=r_1x-\varphi(x,y)$ and introduce a cutoff function $\zeta(s)\in C^{\infty}$ such that
\begin{equation}\label{eqhighorder-a}
\zeta(s)=\begin{cases}
s \quad& \text{for}\,\,s\in (-r_1,r_1),\\[2pt]
0 \quad&\text{for}\,\,s\in \mathbb{R}\backslash (-r_1-1,r_1+1).
\end{cases}
\end{equation}
Then $W(x,y)$ satisfies the following equation:
\begin{align}\label{eqW}
&x\big(r_{1}+\zeta(\frac{W}{x})+\frac{O_{1}}{x}\big)W_{xx}+\big(1+O_{2}\big)W_{yy}-
\big(r_{1}-O_{3}+2r_{1}O_{4}\big)W_{x}+(1+O_{4})W_{x}^{2}\nonumber\\
&=r_{1}O_{3}-r_{1}^{2}O_{4},
\end{align}
where $O_i$, $i=1,\cdots,4$, are given in \eqref{eqhighorder}.

\smallskip
Next, for fixed $z_0=(x_0,y_0)\in Q_{r/2,R/2}^{+}$, we define
$$
W^{(z_0)}(S,T)=\frac{1}{x_{0}^{1+\alpha}}W(x_0+\frac{x_0}{8}S,y_0+\frac{\sqrt{x_0}}{8}T)\qquad \mbox{for $(S,T)\in Q_{1}$},
$$
where $Q_{1}=(-1,1)^2$.
By estimating the coefficients carefully, we can show that equation \eqref{eqW} is uniformly elliptic with
ellipticity constants independent of $z_0$. Then, by Theorem A.1 in Chen-Feldman \cite{chen2010global},
we can derive
$$
\|W^{(z_0)}\|_{C^{2,\alpha}(\overline{Q_{1/2}})}
\leq C \big(r_1 r^{-\alpha}+r^{1-\alpha}\big)=:\widehat{C},
$$
where $C$ depends only on the data and $\alpha$. Thus, we have
$$
|D_{x}^{i}D_{y}^{j}W(x_0,y_0)|\leq Cx_{0}^{2+\alpha-i-j/2} \qquad \text{for all}\,\,(x_0,y_0)\in Q_{r/2,R/2}^{+}, \,
0\leq i+j\leq 1,
$$
which implies that $DW(0,y)=0$. This completes the proof.
\end{proof}

The following lemma states the regularity of solutions near the interior of the sonic boundary.

\begin{lemma}
Let $p\in C^{2,\alpha}(\Omega)\cap C(\overline{\Omega})$ be a solution of Problem {\rm \ref{problem}}
satisfy that
$$
p_{2}<p<p_{1}, \quad r^{2}<p \qquad\,\,\mbox{in $\Omega$}.
$$
Then $p$ cannot be $C^1$ across the degenerate sonic boundary $\Gamma_{\rm sonic}$.
\end{lemma}

\begin{proof}
Suppose that $p$ is $C^{1}$ across $\Gamma_{\rm sonic}$, so is $\varphi=p_{1}-p$.
Since $\varphi\equiv0$ in the supersonic domain $(\{\xi\leq\xi_{1}\}\backslash\Omega)\cap \overline{\Gamma_{\rm sonic}}$,
it follows that
$$
D\varphi(0,y)=0 \qquad\,\, \text{for any}\,\, (0,y)\in\Gamma_{\rm sonic}.
$$
On the other hand, for $(0,y_0)\in \Gamma_{\rm sonic}$ and small $r,R>0$,
$$
\hat{Q}_{r,R}^{+}:=\{(x,y)\,:\, x\in (0,r),\,|y-y_0|<R\}\subset \Omega_{\varepsilon}.
$$
Since \eqref{eqforvarphi} is invariant under the transformation
$(x,y)\rightarrow (x,y-y_0)$,
we can let $(0,y_0)=(0,0)$ so that $Q_{r,R}^{+}\subset \Omega_{\varepsilon}$.
From the proof of Lemma \ref{lem5-2}, there exist $r,\mu>0$ such that
$$
\varphi \geq \mu r_1x\qquad \,\,\text{in $Q_{r,15R/16}^{+}$},
$$
which contradicts to $D\varphi(0,y)=0$.
\end{proof}

For the regularity near the interaction points $P_1$ and $P_3$, we have the following result.

\begin{lemma}\label{theregularity}
Let $p$ be the solution of Problem {\rm \ref{problem}} and satisfy the properties that
there exists a neighborhood $\mathcal{N}(\Gamma_{\rm sonic})$
such that, for $\varphi=p_1-p$,

\begin{enumerate}
\item[\rm (i)] $\varphi$ is $C^{0,1}$ across the degenerate boundary $\Gamma_{\rm sonic}${\rm ;}

\item[\rm (ii)] There exists $\mu_0>0$ such that, in the $(x,y)$--coordinates,
$$
0\leq \varphi \leq (2r_1-\mu_0)x \qquad \text{in $\Omega\cap \mathcal{N}(\Gamma_{\rm sonic})$};
$$

\item[\rm (iii)]  There exist $\epsilon_0>0$, $\omega>0$, and $y=f(x)\in C^{1,1}([0,\epsilon_0])$ such that
\begin{align*}
&\Gamma_{\rm shock}\cap\partial\Omega_{\epsilon_0}=\{(x,y):x\in (0,\epsilon_0),y=f(x)\},\\[1mm]
&\Omega_{\epsilon_0}=\big\{(x,y):x\in (0,\epsilon_0), f(x)<y<\frac{\pi}{2}-\theta_1\big\},\\[1mm]
&\frac{\partial f}{\partial x}\geq \omega>0  \qquad\,\mbox{for $x\in (0,\epsilon_0)$}.
\end{align*}
\end{enumerate}
Then
both limits, $\lim\limits_{(\xi,\eta)\rightarrow P_1\atop (\xi,\eta)\in\Omega}D_{r}\varphi$
and $\lim\limits_{(\xi,\eta)\rightarrow P_3\atop (\xi,\eta)\in\Omega}D_{r}\varphi$, do not exist.
\end{lemma}

\begin{proof}
We prove this assertion by contradiction as in Bae-Chen-Feldman \cite{bae2009regularity}.

We choose two different sequences of points converging to $P_1$ and show that
the two limits $\varphi_{x}$ along the two sequences are different, which reaches to a contradiction.

\smallskip
We first take a sequence close to $\Gamma_{\rm sonic}$.
Let $\{y_{m}\}_{m=1}^{\infty}$ be a sequence such that
$y_{m}\in (0,\frac{\pi}{2}-\theta_1)$ and $\lim\limits_{m\rightarrow \infty}y_{m}=0$.
By \eqref{eq4.9}, there exists
$x_{m}\in (0,\frac{1}{m})$ such that
$$
|\varphi_{x}(x_m,y_m)-r_1|<\frac{1}{m}.
$$
Then we see that $(x_m,y_m)\in \Omega$, $\lim\limits_{m\rightarrow\infty}(x_m,y_m)=0$,
\begin{equation}
\label{eqlimit}
\lim\limits_{m\rightarrow \infty}\varphi_{x}(x_m,y_m)=r_1, \qquad \lim\limits_{m\rightarrow \infty}\varphi_{y}(x_m,y_m)=0.
\end{equation}

\smallskip
We now construct the second sequence close to $\Gamma_{\rm shock}$. Suppose that the limit,
$\lim\limits_{(\xi,\eta)\rightarrow P_1\atop (\xi,\eta)\in\Omega}D_{r}\varphi$, exists.
Then
$$
\lim\limits_{x\rightarrow 0}\varphi(x,f(x))=\varphi(0,0)=0.
$$
From Lemma \ref{lem5-2}, it follows that
$$
\lim\limits_{x\rightarrow 0}\varphi_{y}(x,f(x))=0.
$$
We rewrite the boundary conditions \eqref{eqrh}--\eqref{eqrh2} for $\varphi$ in the $(x,y)$--coordinates as
$$
\widehat{\beta}_{1}\varphi_{x}+\widehat{\beta}_{2}\varphi_{y}=0.
$$
It is direct to see that there exists $\lambda>0$ such that
$\widehat{\beta}_{1}>\lambda$ and $|\widehat{\beta}_{2}|\leq \frac{1}{\lambda}$
on $\Gamma_{\rm shock}\cap\partial\Omega_{\epsilon}$.
Then
$$
|\varphi_{x}(x,f(x))|\leq K|\varphi_{y}(x,f(x))|\qquad\, \text{for some $K>0$},
$$
which implies that $\lim_{x\rightarrow 0}\varphi_{x}(x,f(x))=0$.

\smallskip
Denote $H(x):=\varphi(x,f(x)+\frac{\omega}{2}x)$. Then there exists $\{x_k\}_{k=1}^{\infty}$ with $x_{k}\in (0,\epsilon_0)$
such that
$$
\lim\limits_{k\rightarrow\infty}x_{k}=\lim\limits_{k\rightarrow\infty}H^{\prime}(x_{k})=0.
$$
Moreover, since
$$
H^{\prime}(x)=\varphi_{x}(x,g(x))+\varphi_{y}(x,g(x))g^{\prime}(x),
$$
with $g(x)=f(x)+\frac{\omega}{2}x$ and $|g^{\prime}|\leq K$,
it follows that, for $(x_k,g(x_k))\in \Omega_{\epsilon}$,
\begin{equation} \label{eqlimit2}
\lim\limits_{k\rightarrow\infty}\varphi_{x}(x_k,g(x_k))=0.
\end{equation}
It yields that, along sequence  $(x_k,g(x_k))$, the
limit of $\varphi_{x}(x,y)$ is $0$.
Combining with \eqref{eqlimit},
we conclude that $\varphi_{x}(x,y)$ does not have a limit at $P_1$ from $\Omega$.
Similarly, we can prove that, as a sequence $\{(\xi_i,\eta_i)\}_{i=1}^{\infty}\subset\Omega$ tends to $P_3$,
the limit of $D_{r}\varphi$ does not exist. This completes the proof.
\end{proof}

\section{Existence and Regularity of Global Solutions of the Pressure Gradient System}
\label{sec6}

In Theorem \ref{thm1}, we have constructed a global solution $p$ of the second order equation \eqref{eqforp} in $\Omega$, which is piecewise
constant in the supersonic region.
Moreover, we have proved that $p$ is Lipschitz continuous across the degenerate sonic boundary $\Gamma_{\rm sonic}$ from $\Omega$ to the supersonic region.

To recover the velocity components $u$ and $v$, we consider the first two equations in \eqref{pgs2}.
We can write these equations in the radial variable $r$ as
$$
\frac{\partial u}{\partial r}=\frac{1}{r}p_{\xi},\qquad  \frac{\partial v}{\partial r}
=\frac{1}{r}p_{\eta},
$$
and integrate from the boundary of the subsonic region toward the origin.
It is direct to see that $(u,v)$ are at least Lipschitz continuous across $\Gamma_{\rm sonic}$.
Furthermore, $(u,v)$ have the same regularity as $p$ inside $\Omega$ except origin $r=0$.
However, $(u,v)$ may be multi-valued at origin $r=0$.

In conclusion, we have
\begin{theorem}\label{thm2}
Let the Riemann initial data satisfy \eqref{initialdate1}.
Then there exists a global solution $(u,v,p)(r,\theta)$ with the free boundary
$r=r(\theta), \theta\in [\theta_{3},\theta_{1}]$, such that
\begin{equation*}
(u,v,p)\in C^{2,\alpha}(\Omega), \quad p\in C^{\alpha}(\overline{\Omega}), \quad
r\in C^{2,\alpha}((\theta_{3},\theta_{1}))\cap C^{1,1}([\theta_{3},\theta_{1}]),
\end{equation*}
and $(u,v,p)$ are piecewise constant in the supersonic region.
Moreover, the global solution $(u,v,p)$ with the free boundary $r=r(\theta)$
satisfies the following properties{\rm :}

\smallskip
\emph{(i)} $p>p_2$ on $\Gamma_{\rm shock}$; that is, shock $\Gamma_{\rm shock}$ does not meet the sonic circle of state $p_2${\rm ;}

\smallskip
\emph{(ii)} The shock, $\Gamma_{\rm shock}$, is strictly convex in the self-similar coordinates{\rm ;}

\smallskip
\emph{(iii)} The solution, $(u,v,p)$, is $C^{0,\alpha}$ up to the sonic boundary $\Gamma_{\rm sonic}$ and Lipschitz continuous across $\Gamma_{\rm sonic}${\rm ;}

\emph{(iv)} The Lipschitz regularity of both solution $(u,v,p)$ across $\Gamma_{\rm sonic}$ from the subsonic region $\Omega$
and shock $\Gamma_{\rm shock}$ across points $\{P_1, P_3\}$
is optimal.
\end{theorem}

\bigskip
\medskip
\noindent
{\bf Acknowledgements}.
Gui-Qiang G. Chen's research was supported in part by the UK Engineering and Physical Sciences Research Council
under Grant EP/L015811/1 and the Royal Society--Wolfson Research Merit Award WM090014 (UK).
Qin Wang's research was supported in part by National Natural Science Foundation of China (11761077),
China Scholarship Council (201807035046), and the Key Project of Yunnan Provincial Science and Technology
Department and Yunnan University (No.2018FY001(-014)).
Shengguo Zhu's research was supported in part by the Royal Society--Newton International Fellowships NF170015,
and  Monash University--Robert Bartnik Visiting Fellowship.
Qin Wang would also like to thank the hospitality and support
of Mathematical Institute of University of Oxford during his visit in 2019--20.


\bigskip
\begin{thebibliography}{10}

\bibitem{agarwal1994modified}
R.~Agarwal and D.~Halt.
\newblock A modified CUSP scheme in wave/particle split form for unstructured
  grid Euler flows.
\newblock In: {\em Frontiers of Computational Fluid Dynamics}, Eds. D.~A. Caughey and M.~M. Hafez,
pp. 155--163, 1994.

\bibitem{bae2009regularity}
M.~Bae, G.-Q. Chen, and M.~Feldman.
\newblock Regularity of solutions to regular shock reflection for potential
  flow.
\newblock {\em Invent. Math.} {\bf 175(3)}:505--543, 2009.

\bibitem{bae2019prandtl}
M.~Bae, G.-Q. Chen, and M.~Feldman.
\newblock {\em Prandtl-Meyer Reflection Configurations, Transonic Shocks, and Free
Boundary Problems}.  Research Monograph, 224 pages,\,
Memoirs of the American Mathematical Society, AMS: Providence, 2020 (to appear).
{\em arXiv Preprint}, arXiv:1901.05916.

\bibitem{vcanic2000free}
S.~Canic, B.~L. Keyfitz, and E.~H. Kim.
\newblock Free boundary problems for the unsteady transonic small disturbance
  equation: Transonic regular reflection.
\newblock {\em Methods  Appl. Anal.} {\bf 7(2)}:313--336, 2000.

\bibitem{canic2001free}
S.~Canic, B.~L. Keyfitz, and E.~H. Kim.
\newblock A free boundary problem for a quasi-linear degenerate elliptic
  equation: regular reflection of weak shocks.
\newblock {\em Commun. Pure Appl. Math.} {\bf 55(1)}:71--92,
  2002.

\bibitem{canic2006free}
S.~Canic, B.~L. Keyfitz, and E.~H. Kim.
\newblock Free boundary problems for nonlinear wave systems: Mach stems for
  interacting shocks.
\newblock {\em SIAM J. Math. Anal.} {\bf 37(6)}:1947--1977, 2006.

\bibitem{ChangChenYang}
T.~Chang, G.-Q. Chen, and S.-L. Yang.
On the $2$-D Riemann problem for the compressible Euler equations. I. Interaction of shocks and rarefaction waves.
\newblock{\em Discrete Contin. Dynam. Systems}, {\bf 1}: 555--584, 1995.
II. Interaction of contact discontinuities.
\newblock{\em Discrete Contin. Dynam. Systems}, {\bf 6}: 419--430, 2000.

\bibitem{chang1989riemann}
T.~Chang and L.~Hsiao.
\newblock {\em The Riemann Problem and Interaction of Waves in Gas Dynamics},
Longman Scientific \& Technical: Harlow; John Wiley \& Sons, Inc.: New York, 1989.


\bibitem{chen2014shock}
G.-Q. Chen, X.~Deng, and W.~Xiang.
\newblock Shock diffraction by convex cornered wedges for the nonlinear wave
  system.
\newblock {\em Arch. Ration. Mech. Anal.} {\bf 211(1)}:61--112,
  2014.

\bibitem{chen2010global}
G.-Q. Chen and M.~Feldman.
\newblock Global solutions of shock reflection by large-angle wedges for
  potential flow.
\newblock {\em  Ann. of Math.} {\bf (2) 171}:1067--1182, 2010.

\bibitem{chen2018mathematics}
G.-Q. Chen and M.~Feldman.
\newblock {\em The Mathematics of Shock Reflection-Diffraction and Von
  Neumann's Conjectures}, Research Monograpgh,
 Annals of Mathematics Studies, Vol. {\bf 359}.
\newblock Princeton University Press: Princeton, 2018.

\bibitem{ChenLeFloch}
G.-Q. Chen and P. LeFloch.
\newblock Entropy flux-splittings for hyperbolic conservation laws.
\newblock {\em Comm. Pure Appl. Math.} {\bf 48}: 691--729, 1995.

\bibitem{SXChen1992}
S.-X.~Chen.
Multidimensional Riemann problem for semilinear wave equations.
\newblock {\em Comm. Partial Diff. Equ.} {\bf 17}: 715--736, 1992.

\bibitem{SXChen1997}
S.-X.~Chen.
Construction of solutions to M-D Riemann problems for a $2\times 2$ quasilinear hyperbolic system.
\newblock {\em Chinese Ann. Math. Ser. B}, {\bf 18}: 345--358, 1997.


\bibitem{chen2007stability}
S.-X.~Chen and B.~Fang.
\newblock Stability of transonic shocks in supersonic flow past a wedge.
\newblock {\em J. Differ. Equ.} {\bf 233(1)}:105--135, 2007.

\bibitem{Chen-Qu2012}
S.-X. Chen and A. Qu. Two-dimensional Riemann problems for Chaplygin gas.
\newblock {\em SIAM J. Math. Anal.} {\bf 44}: 2146--2178, 2012.


\bibitem{dafermos2000}
C.~M. Dafermos.
\newblock {\em Hyperbolic Conservation laws in Continuum Physics},
\newblock Springer-Verlag: Berlin, 2016.

\bibitem{egorov1969oblique}
J.~V. Egorov and V.~A. Kondrat'ev.
\newblock The oblique derivative problem.
\newblock {\em Mathematics of the USSR-Sbornik}, {\bf 7(1)}:139, 1969.

\bibitem{elling2008supersonic}
V.~Elling and T.-P. Liu.
\newblock Supersonic flow onto a solid wedge.
\newblock {\em Commun. Pure Appl. Math.} {\bf 61(10)}:1347--1448, 2008.


\bibitem{gilbarg2015elliptic}
D.~Gilbarg and N.~S. Trudinger.
\newblock {\em Elliptic Partial Differential Equations of Second Order}.
\newblock Springer-Verlag: Berlin, 2001.

\bibitem{glimm1965}
J.~Glimm.
\newblock  Solutions in the large for nonlinear hyperbolic systems of equations.
\newblock{\em  Comm. Pure Appl. Math.} {\bf 18}:697--715, 1965.

\bibitem{hormander1966pseudo}
L.~Hormander.
\newblock Pseudo-differential operators and non-elliptic boundary problems.
\newblock {\em Ann. of Math.}  {\bf 83}:129--209, 1966.

\bibitem{keyfitz1998riemann}
B.~L. Keyfitz and S.~Canic.
\newblock Riemann problems for the two-dimensional unsteady transonic small
  disturbance equation.
\newblock {\em SIAM J. Appl. Math.} {\bf 58(2)}:636--665, 1998.

\bibitem{kim2010global}
E.~H. Kim.
\newblock A global subsonic solution to an interacting transonic shock for the
  self-similar nonlinear wave equation.
\newblock {\em J. Differ. Equ.} {\bf 248(12)}:2906--2930, 2010.


\bibitem{lax1957hyperbolic}
P.~Lax.
\newblock Hyperbolic systems of conservation laws II.
\newblock {\em Commun. Pure Appl. Math.} {\bf 4(10)}:537--566, 1957.

\bibitem{li1998two}
J.~Li, T.~Zhang, and S.~Yang.
\newblock {\em The Two-Dimensional Riemann Problem in Gas Dynamics}.
Monographs and Surveys in Pure and Applied Mathematics, Vol. {\bf 98},
\newblock Chapman \& Hall/CRC, Longman: Harlow, 1998.

\bibitem{li1985second}
Y.~F. Li and Y.~M. Cao.
\newblock  {\em Large-particle} difference method with second-order accuracy in gasdynamics.
{\em Sci. China}, {\bf 28A}:1024--1035, 1985.


\bibitem{lieberman1985perron}
G.~M. Lieberman.
\newblock The Perron process applied to oblique derivative problems.
\newblock {\em Adv. Math.} {\bf 55(2)}:161--172, 1985.

\bibitem{lieberman1986mixed}
G.~M. Lieberman.
\newblock Mixed boundary value problems for elliptic and parabolic differential
  equations of second order.
\newblock {\em J. Math. Anal. Appl.}
  {\bf 113(2)}:422--440, 1986.

\bibitem{popivanov1997degenerate}
P.~R. Popivanov, and D.~K. Palagachev.
\newblock {\em The Degenerate Oblique Derivative Problem for Elliptic and
  Parabolic Equations}.
\newblock Akademie Verlag: Berlin, 1997.

\bibitem{Riemann}
B. Riemann. \"{U}ber die Fortpflanzung ebener Luftvellen von
endlicher Schwingungsweite. G\"{o}tt. Abh. Math. Cl. {\bf 8}:43--65, 1860.


\bibitem{smoller1994}
J.~Smoller.
\newblock {\em The Shock Waves and Reaction-Diffusion Equations}.
\newblock  2nd edition, Springer-Verlag: New York, 1994.

\bibitem{winzell1981boundary}
B.~Winzell.
\newblock A boundary value problem with an oblique derivative.
\newblock {\em Commun. Partial Differ. Equ.}
  {\bf 6(3)}:305--328, 1981.

\bibitem{yuan2006transonic}
H.~Yuan.
\newblock On transonic shocks in two-dimensional variable-area ducts for steady
  Euler system.
\newblock {\em SIAM J. Math. Anal.} {\bf 38(4)}:1343--1370, 2006.

\bibitem{zhang1998two}
P.~Zhang, J.~Li, and T.~Zhang.
\newblock On two-dimensional Riemann problem for pressure-gradient equations of
  the Euler system.
\newblock {\em Discret. Contin. Dyn. Syst.} {\bf 4}:609--634, 1998.

\bibitem{yuxi1997existence}
Y.~Zheng.
\newblock Existence of solutions to the transonic pressure gradient equations
  of the compressible Euler equations in elliptic regions.
\newblock {\em Commun. Partial Differ. Equ.}
  {\bf 22(11-12)}:1849--1868, 1997.


\bibitem{zheng2003global}
Y.~Zheng.
\newblock A global solution to a two-dimensional Riemann problem involving
  shocks as free boundaries.
\newblock {\em Acta Math. Appl. Sin.} {\bf 19(4)}:559--572, 2003.

\bibitem{zheng2006two}
Y.~Zheng.
\newblock Two-dimensional regular shock reflection for the pressure gradient
  system of conservation laws.
\newblock {\em Acta Math. Appl. Sin.} {\bf 22(2)}:177--210, 2006.

\bibitem{zheng2012systems}
Y.~Zheng.
\newblock {\em Systems of Conservation Laws{\rm :} Two-Dimensional Riemann Problems}.
  vol. {\bf 38},
\newblock Springer Science \& Business Media, 2012.

\end{thebibliography}
\end{document}